\documentclass{amsart}

\usepackage[all]{xy}

\usepackage{hyperref}

\usepackage{cleveref}

\usepackage{amssymb}

\theoremstyle{plain}
\newtheorem{theorem}[subsection]{Theorem}
\newtheorem{proposition}[subsection]{Proposition}
\newtheorem{lemma}[subsection]{Lemma}
\newtheorem{corollary}[subsection]{Corollary}
\theoremstyle{definition}
\newtheorem{definition}[subsection]{Definition}
\newtheorem{example}[subsection]{Example}

\newcommand{\ad}{\mathrm{ad}}
\newcommand{\Aut}{\mathrm{Aut}}
\newcommand{\Bun}{\mathrm{Bun}}
\newcommand{\crys}{\mathrm{crys}}
\newcommand{\dR}{\mathrm{dR}}
\newcommand{\F}{\mathbb{F}}

\newcommand{\Hom}{\mathrm{Hom}}
\newcommand{\Id}{\textrm{Id}}

\renewcommand{\inf}{\mathrm{inf}}

\newcommand{\Q}{\mathbb{Q}}
\newcommand{\Rep}{\mathrm{Rep}}
\newcommand{\rk}{\mathrm{rk}}
\newcommand{\Spa}{\mathrm{Spa}}
\newcommand{\Spec}{\textrm{Spec}}
\newcommand{\Z}{\mathbb{Z}}

\begin{document}

\title{Breuil-Kisin-Fargues modules with complex multiplication}
\author{Johannes Ansch\"{u}tz}
\email{ja@math.uni-bonn.de}
\date{\today}

\begin{abstract}
We prove that the category of (rigidified) Breuil-Kisin-Fargues modules up to isogeny is Tannakian. We then introduce and classify Breuil-Kisin-Fargues modules with complex multiplication mimicking the classical theory for rational Hodge structures. In particular, we compute an avatar of a ``p-adic Serre group''.
\end{abstract}

\maketitle

\section{Introduction}

In \cite{fargues_quelques_resultats_et_conjectures_concernant_la_courbe} L.\ Fargues introduced an analog, called Breuil-Kisin-Fargues modules, of Breuil-Kisin modules (cf.\  \cite{kisin_crystalline_representations_and_f_crystals}) over Fontaine's first period ring
$$
A_\inf=W(\mathcal{O}_{C^\flat})
$$
of Witt vectors of the ring of integers $\mathcal{O}_{C^\flat}$ inside $C^\flat$ where $C$ denotes a non-archimedean, complete and algebraically closed extension of $\Q_p$ and 
$$
C^\flat=\varprojlim\limits_{x\mapsto x^p}C
$$ 
its tilt.
To define Breuil-Kisin-Fargues modules let $\xi\in A_\inf$ be a generator of the kernel of Fontaine's map
$$
\theta\colon A_\inf\to \mathcal{O}_C
$$
(cf.\ \Cref{section: breuil_kisin_fargues_modules_up_to_isogeny})
and let
$$
\varphi\colon A_\inf\to A_\inf
$$ 
be the Frobenius of $A_\inf$ (induced by the Frobenius of $\mathcal{O}_{C^\flat}$).
Concretely, a Breuil-Kisin-Fargues modules $(M,\varphi_M)$ is then a finitely presented $A_\inf$-module $M$ together with an isomorphism
$$
\varphi_M\colon \varphi^\ast(M)[\frac{1}{\varphi(\xi)}]\cong M[\frac{1}{\varphi(\xi)}]
$$
such that $M[\frac{1}{p}]$ is finite projective over $A_\inf[\frac{1}{p}]$ (cf.\ \Cref{definition: bkf_module}).
The study of Breuil-Kisin-Fargues modules, which are mixed-characteristic analogs of Drinfeld's shtuka, was taken over in \cite{scholze_weinstein_lecture_notes_on_p_adic_geometry} and \cite{bhatt_morrow_scholze_integral_p_adic_hodge_theory}. More precisely, in \cite{bhatt_morrow_scholze_integral_p_adic_hodge_theory} to every proper smooth formal scheme $\mathfrak{X}$ over $\mathcal{O}_C$ and every $i\geq 0$ there is associated a Breuil-Kisin-Fargues module
$$
H^i_{A_\inf}(\mathfrak{X})
$$
interpolating, at least rationally, various cohomology groups attached to $\mathfrak{X}$. Namely, Breuil-Kisin-Fargues modules admit various realizations (cf.\ \Cref{definition: realizations_of_bkf_modules}):
Let $(M,\varphi_M)$ be a Breuil-Kisin-Fargues module and denote by $k$ the residue field of $\mathcal{O}_C$. We can associate to $M$
\begin{itemize}
\item its ``\'etale realization'' 
$$
T:=(M\otimes_{A_\inf}W(C^\flat))^{\varphi_M=1},
$$ 
a finitely generated $\Z_p$-module,
\item its ``crystalline realization'' 
$$
D:=M\otimes_{A_\inf}W(k),
$$ 
a finitely generated $W(k)$-module equipped with a $\varphi$-semilinear isomorphism $\varphi_D:\varphi^\ast(D[\frac{1}{p}])\cong D[\frac{1}{p}]$ after inverting $p$,
\item and its ``de Rham realization'' 
$$
V:=M\otimes_{A_\inf,\theta}\mathcal{O}_C,
$$ a finitely generated $\mathcal{O}_C$-module. 
\end{itemize}
If $(M,\varphi)=H^i_{A_\inf}(\mathfrak{X})$ for a proper smooth formal scheme $\mathfrak{X}/\mathcal{O}_C$, then the main result of \cite{bhatt_morrow_scholze_integral_p_adic_hodge_theory} shows that these various realizations are, at least after inverting $p$, given by the \'etale cohomology $H^i_{\acute{e}t}(X,\Q_p)$ of the generic fiber $X:=\mathfrak{X}_C$ of $\mathfrak{X}$, the crystalline cohomology $H^i_{\crys}(\mathfrak{X}_0/W(k))[\frac{1}{p}]$ of the special fiber $\mathfrak{X}_0\subseteq \mathfrak{X}$, and the deRham cohomology $H^i_{\dR}(X/C)$ of $X$.
Thus we see that Breuil-Kisin-Fargues modules have a ``motivic flavour''.
But the general picture surrounding motives is looking for the existence of a tensor functor from smooth projective schemes into some \textit{Tannakian} category, and the category of Breuil-Kisin-Fargues modules is not Tannakian. In fact, it is even not abelian (cf.\ \cite[Remark 4.25.]{bhatt_morrow_scholze_integral_p_adic_hodge_theory}). To remedy this we follow an idea in \cite{bhatt_morrow_scholze_integral_p_adic_hodge_theory} and introduce \textit{rigidifications} of Breuil-Kisin-Fargues modules.
Fix a section $k\to \mathcal{O}_C/p$ of the projection $\mathcal{O}_C/p\to k$.
If $(M,\varphi_M)$ is a Breuil-Kisin-Fargues module then a rigidification for $(M,\varphi_M)$ (cf.\ \cite[Lemma 4.27.]{bhatt_morrow_scholze_integral_p_adic_hodge_theory} and \Cref{definition: rigidified_bkf_modules}) is an isomorphism
$$
\alpha\colon M\otimes_{A_\inf}B^+_\crys\cong (M\otimes_{A_\inf}W(k))\otimes_{W(k)}B^+_\crys
$$
of $\varphi$-modules over $B^+_\crys$ inducing the identity when base changed to $W(k)[\frac{1}{p}]$. The first main theorem of this paper is the following (the result is already stated in \cite[Remark 4.25.]{bhatt_morrow_scholze_integral_p_adic_hodge_theory}).

\begin{theorem}[cf.\ \Cref{theorem: rigidified_bkf_modules_abelian}]
\label{theorem:rigidified-bkf-modules-abelian-intro}
The category $$\mathrm{BKF}_{\mathrm{rig}}$$ of rigidified Breuil-Kisin-Fargues modules is abelian.   
\end{theorem}

We remark that the analogous statement for Breuil-Kisin modules is true, but much simpler. The problems for Breuil-Kisin-Fargues modules arise as the ring $A_\inf$ is highly non-noetherian. But luckily we can profit from \cite{bhatt_morrow_scholze_integral_p_adic_hodge_theory}, where enough commutative algebra of $A_\inf$-modules is developed.
From \Cref{theorem:rigidified-bkf-modules-abelian-intro} it is not difficult to deduce that the $\Q_p$-linear category
$$
\mathrm{BKF}_{\mathrm{rig}}^\circ:=\Q_p\otimes_{\Z_p}\mathrm{BKF}_{\mathrm{rig}}
$$
of rigidified Breuil-Kisin-Fargues modules up to isogeny is Tannakian (cf.\ \Cref{theorem: rigidified_bkf_modules_form_tannakian_category}). The statement is known for Breuil-Kisin modules, and again it is much simpler.
We investigate the Tannakian category $\mathrm{BKF}_{\mathrm{rig}}^\circ$ a bit: for example, we prove that it is ``connected'' (cf.\ \Cref{lemma:band-of-bkf-modules-connected}) and of homological dimension $1$ (cf.\ \Cref{lemma:homological-dimension}).

In \Cref{section: cm_breuil_kisin_fargues_modules} we start to classify rigidified Breuil-Kisin-Fargues modules admitting ``complex multiplication'', i.e., Breuil-Kisin-Fargues modules whose Mumford-Tate group in the Tannakian category $\mathrm{BKF}_{\mathrm{rig}}^\circ$ is a torus (cf.\ \Cref{definition:cm-object} and \Cref{lemma:characterization-of-cm-objects}). More concretely, a rigidified Breuil-Kisin-Fargues module $(M,\varphi,\alpha)$ up to isogeny admits CM if and only if there exists a commutative, semisimple $\Q_p$-algebra $E$ and an injection $E\hookrightarrow \mathrm{End}_{\mathrm{BKF}_{\mathrm{rig}}^\circ}((M,\varphi,\alpha))$ of $\Q_p$-algebras such that
$$
\mathrm{dim}_{\Q_p}E=\mathrm{rk}_{A_\inf}(M)
$$ 
(cf.\ \Cref{lemma:characterization-of-cm-objects}).
Using the crucial theorem about different descriptions of (finite free) Breuil-Kisin-Fargues modules (cf.\ \Cref{theorem: equivalent_descriptions_of_bkf_modules}) due to Fargues/Kedlaya-Liu/Scholze we can then prove our main theorem about the classification of Breuil-Kisin-Fargues modules admitting CM.

\begin{theorem}[cf.\ \Cref{lemma:integral-cm-bkf-modules}]
\label{theorem:classification-cm-bkf-modules-intro}
For every finite dimensional commutative, semisimple $\Q_p$-algebra $E$ there exists a (natural) bijection between isomorphism classes of (rigidified) Breuil-Kisin-Fargues modules (up to isogeny) admitting CM by $E$ and functions $\Phi\colon \mathrm{Hom}_{\Q_p}(E,C)\to \Z$.
\end{theorem}

Moreover, we can write down for given pair $(E,\Phi\colon\mathrm{Hom}_{\Q_p}(E,C)\to \Z)$ the corresponding (rigidified) Breuil-Kisin-Fargues module (up to isogeny) explicitly (cf.\ \Cref{theorem-rigidified-bkf-modules-with-cm}). 
We remark that a similar result has been obtained by Lucia Mocz in the case of Breuil-Kisin modules associated with $p$-divisible groups (cf.\ \cite{mocz_a_new_northcott_property}).

Finally, let $\mathcal{T}\subseteq \mathrm{BKF}^\circ_{\mathrm{rig}}$ be the full Tannakian subcategory spanned by rigidified Breuil-Kisin-Fargues modules up to isogeny admitting CM and let 
$$
D_{\Q_p}
$$
be the pro-torus over $\Q_p$ with group of characters the coinduced discrete Galois module
$$
X^\ast(D_{\Q_p})=\mathrm{Coind}^{\mathrm{Gal}(\overline{\Q}_p/\Q_p)}\Z
$$
(cf.\ \Cref{lemma-character-group-of-pro-torus}). We prove the following description of the category $\mathcal{T}$ of Breuil-Kisin-Fargues modules admitting CM.

\begin{theorem}[cf.\ \Cref{proposition-tannakian-group-for-cm-bkf-modules}]
\label{theorem:tannakian-group-of-cm-bkf-modules}
The \'etale realization defines an equivalence
$$
\mathcal{T}\cong \mathrm{Rep}_{\Q_p}(D_{\Q_p})
$$  
of Tannakian categories.
\end{theorem}

This theorem can be understood as the computation of a ``$p$-adic Serre group'', analogous to the case of rational Hodge structures (cf.\ \cite{fargues_abelian_motives}).

In order to prove \Cref{theorem:tannakian-group-of-cm-bkf-modules} we develop in \Cref{section: formal_cm_theory} some language concerning CM-objects and reflex norms in an arbitrary Tannakian category to formalize the known case of rational Hodge structures admitting CM.

\subsection*{Acknowledgement} The author wants to thank Lucia Mocz heartily for sharing her notes \cite{mocz_a_new_northcott_property} on Breuil-Kisin modules with CM, which eventually led to the question of defining and classifying Breuil-Kisin-Fargues modules with CM answered in this paper. Moreover, the author wants to thank Bhargav Bhatt, Peter Scholze and Sebastian Posur for discussions surrounding this paper. Especial thanks go to Peter Scholze for providing the hint to the remark following \cite[Lemma 27.]{bhatt_morrow_scholze_integral_p_adic_hodge_theory} (i.e., \Cref{theorem: rigidified_bkf_modules_abelian}) which lead to the construction of the Tannakian category $\mathrm{BKF}^\circ_{\mathrm{rig}}$.

\section{Formal CM-theory}
\label{section: formal_cm_theory}

In this section we write down a general theory of ``CM-objects'' in a Tannakian category $\mathcal{T}$, mainly fixing terminology. Concretely this means that we reformulate the maximal torus quotient of the band of the category $\mathcal{T}$ internal in $\mathcal{T}$. We will apply this theory to the case that $\mathcal{T}=\mathrm{BKF}^\circ_{\mathrm{rig}}$ is the category of rigidified Breuil-Kisin-Fargues modules up to isogeny (cf.\ \Cref{definition: rigidified_bkf_modules_up_to_isogeny}).
The formalism of this section is a straightforward translation of CM-theory for rational Hodge structures or categories of (CM) motives to the case of general Tannakian categories. Therefore we expect it to be known in principle and do not claim any originality.
We advise the reader to have a look at \cite{fargues_abelian_motives}.

Let $k$ be a field (later it will assumed to have characteristic $0$) and let $\mathcal{T}$ be a Tannakian category over $k$ (not necessarily assumed to be neutral). For an object $X\in \mathcal{T}$ we denote by $\langle X\rangle^{\otimes}$ the full Tannakian subcategory spanned by $X$.

\begin{definition}
\label{definition:cm-object}
An object $X\in \mathcal{T}$ is called a CM-object, or to admit CM, if the connected component of the band of the Tannakian category $\langle X\rangle^{\otimes}$ is multiplicative, i.e., for every fiber functor
$$
\omega\colon \langle X\rangle^{\otimes}\to \mathrm{Bun}_{S},
$$
where $S/k$ is a scheme, the connected component $G^\circ$ of the group scheme
$$
G:=\mathrm{Aut}^{\otimes}(\omega)
$$
is a multiplicative group scheme over $S$.   
\end{definition}

Equivalently, the condition can be required only for the case $S=\Spec(k^\prime)$ is the spectrum of a field extension $k^\prime/k$, or even only for one fiber functor over some field extension $k^\prime/k$.
Moreover, if $k^\prime/k$ is a finite field extension and $\mathcal{T}_{k^\prime}$ the base change of $\mathcal{T}$ from $k$ to $k^\prime$ (cf.\ \cite[Section 2.5.]{ziegler_graded_and_filtered_fiber_functors}), then $X\in \mathcal{T}$ is a CM-object if and only if $k^\prime\otimes_k X\in \mathcal{T}_{k^\prime}$ is a CM-object, because a fiber functor $\omega\colon \langle X\rangle^\otimes \to \mathrm{Vec}_{k^{\prime\prime}}$ where $k^{\prime\prime}$ is a field extension of $k$ containing $k^\prime$ extends to a fiber functor $\omega^\prime\colon \langle k^\prime\otimes_k X\rangle^\otimes$ (cf.\ \cite[Section 2.5.]{ziegler_graded_and_filtered_fiber_functors}).  

If $k$ is of characteristic $0$, then a connected multiplicative group scheme of finite type is automatically a torus. 

Definition~\Cref{definition:cm-object} is a formalization of the definition of a CM-rational Hodge structure (cf.\ \cite{griffiths_mumford_tate_groups}). Namely, let $V$ be a (polarisable) rational Hodge structure and let $\mathbb{S}:=\mathrm{Res}_{\mathbb{C}/\mathbb{R}}\mathbb{G}_m$ be the Deligne torus. Then $V$ is called of $CM$-type if the ``Mumford-Tate group'' of $V$, i.e., the minimal closed subgroup $G\subseteq \mathrm{GL}(V)$ over $\Q$ containing the image of the morphism $h\colon \mathbb{S}\to \mathrm{GL}(V)$ induced by the Hodge structure $V_{\mathbb{C}}=\bigoplus\limits_{p,q\in \Z}V^{p,q}$ on $V$, is a torus (cf.\ \cite{griffiths_mumford_tate_groups}).
This definition agrees with ours as the Mumford-Tate group of $V$ is precisely the automorphism group of the canonical fiber functor
$$
\omega\colon \langle V\rangle^{\otimes}\to \mathrm{Vec}_{\Q}. 
$$

\begin{lemma}
\label{lemma:representations-of-tori-are-cm-objects}
Let $T/k$ be an affine group scheme of finite type which is an extension of a finite discrete group $G$ and multiplicative group $T^0$, i.e., there exists an exact sequence (of fppf-sheaves)
$$
1\to T^0\to T\to G\to 1.
$$
Then every $V\in \Rep_k(T)$ is a CM-object.   
\end{lemma}
\begin{proof}
Let $V\in \Rep_k(T)$ and set $\mathcal{T}:=\langle V\rangle^\otimes$ to be the Tannakian subcategory generated by $V$. Let $\omega\colon \mathcal{T}\to \mathrm{Vec}_k$ be the restriction of the canonical fiber functor. Then the canonical morphism
$$
T\to \mathrm{Aut}^\otimes(\omega)
$$  
is faithfully flat and of finite presentation. This implies that this morphism is open and hence the connected component $T^0$ of $T$ surjects onto the connected component $\mathrm{Aut}^\otimes(\omega)^\circ$ of $\mathrm{Aut}^\otimes(\omega)$. In particular, the group $\mathrm{Aut}^\otimes(\omega)^\circ$ is multiplicative, i.e., $V$ is a CM-object. 
\end{proof}

\begin{lemma}
\label{lemma:subquotients-of-cm-objects-again-cm}
Let $\mathcal{T}$ be a Tannakian category over $k$ and let $X\in\mathcal{T}$ be a CM-object. Then every $Y\in \langle X\rangle^\otimes$ is again a CM-object.
Moreover, if $X,Y\in \mathcal{T}$ are CM-objects, then $X\oplus Y$ and $X\otimes Y$ are CM-objects as well.  
\end{lemma}
\begin{proof}
For the first statment we may replace $\mathcal{T}$ by $\langle X\rangle^\otimes$, and, after enlarging $k$, assume moreover that $\mathcal{T}$ is neutral. As $X$ is a CM-object, the category $\mathcal{T}\cong \Rep_k(T)$ is thus equivalent to the category of representations of an affine group scheme $T$, which is an extension of a finite discrete group by a multiplicative group as in \Cref{lemma:representations-of-tori-are-cm-objects}. \Cref{lemma:representations-of-tori-are-cm-objects} implies that $Y\in \mathcal{T}=\langle X\rangle^\otimes$ is again a CM-object. 
Now let $X,Y\in \mathcal{T}$ (with $\mathcal{T}$ arbitrary) be CM-objects. Then $X\otimes Y\in \langle X\oplus Y\rangle^\otimes$ and thus it suffices to proof that $X\oplus Y$ is again a CM-object by what has already been shown.
But if 
$$
\omega\colon \langle X\oplus Y\rangle^\otimes\to \mathrm{Bun}_{S}
$$
is a fiber functor, and we let $\omega_1$, resp.\ $\omega_2$, be the restrictions of $\omega$ to $\langle X\rangle^\otimes\subseteq \langle X\oplus Y\rangle^\otimes$, resp.\ $\langle Y\rangle^\otimes\subseteq \langle X\oplus Y\rangle^\otimes$, then the canonical morphism
$$
\mathrm{Aut}^\otimes(\omega)\hookrightarrow \mathrm{Aut}^\otimes(\omega_1)\times \mathrm{Aut}^\otimes(\omega_2)
$$
is a closed immersion. In particular, the connected component of $\mathrm{Aut}^\otimes(\otimes)$ is again multiplicative.  
\end{proof}

\begin{definition}
\label{definition:subcategory-of-cm-objects}
Let $\mathcal{T}$ be a Tannakian category over $k$.
We denote by $\mathcal{T}_{\mathrm{CM}}\subseteq \mathcal{T}$ the full subcategory of CM-objects of $\mathcal{T}$.
\end{definition}

By \Cref{lemma:subquotients-of-cm-objects-again-cm}, the category $\mathcal{T}_{\mathrm{CM}}$ is a Tannakian subcategory. In general, it is not closed under extensions, e.g., for representations of unipotent groups.

\begin{lemma}
\label{lemma: exact_tensor_functors_preserve_cm_objects}
Let $\eta\colon \mathcal{T}\to \mathcal{T}^\prime$ be an exact tensor functor and let $X\in \mathcal{T}$ be a CM object. Then $\eta(X)\in \mathcal{T}^\prime$ is a CM object. In particular, $\eta$ induces an exact tensor functor $\eta\colon \mathcal{T}_{\mathrm{CM}}\to \mathcal{T}^\prime_{\mathrm{CM}}$ on the full subcategories of CM-objects.   
\end{lemma}
\begin{proof}
Let $\omega^\prime\colon \langle \eta(X)\rangle^{\otimes}\to \mathrm{Vec}_{k^\prime}$ 
be a fiber functor where $k^\prime/k$ is a field extension. Set $\omega:=\omega^\prime\circ \eta$. The morphism
$$
\eta^\ast\colon H:=\mathrm{Aut}(\omega^\prime)\to G:=\mathrm{Aut}(\omega)
$$
is then injective. As the connected component of $G$ is multiplicative, the connected component $H^\circ$ of $H$ will therefore be multiplicative as well.
\end{proof}

The following direct corollary can be useful to prove that certain Breuil-Kisin-Fargues modules admit CM.

\begin{corollary}
\label{corollary: constructing_objects_with_cm_using_tori}
Let $T/k$ be a torus and let $\omega\colon \Rep_k(T)\to \mathcal{T}$ be an exact tensor functor. Then for every object $V\in \Rep_k(T)$ the object $\omega(V)$ has CM.  
\end{corollary}
\begin{proof}
This is a special case of \Cref{lemma: exact_tensor_functors_preserve_cm_objects} as any object in the category $\mathrm{Rep}_k(T)$ of representations of $T$ admits CM (cf.\ \Cref{lemma:representations-of-tori-are-cm-objects}).  
\end{proof}

We now want to give a more explicit definition of CM objects.
From now on assume that $k$ has characteristic $0$.
Recall that every object $X\in \mathcal{T}$ in a Tannakian category $\mathcal{T}$ has a rank
$$
\mathrm{rk}(X)\in \mathrm{End}(1_\mathcal{T})\cong k
$$
defined as the trace of the identity on $X$. In general, the rank is an endomorphism of the unit object $1_\mathcal{T}$.
But as we assumed that $k$ is of characteristic $0$, the rank of $X$ equals the dimension (over $k^\prime$) of $\omega(X)$ for one, or equivalently any, fiber functor $\omega\colon \mathcal{T}\to \mathrm{Vec}_{k^\prime}$ for $k^\prime/k$ some field extension.

We call a Tannakian category $\mathcal{T}$ connected if its band is connected. For a neutral Tannakian category $\mathcal{T}\cong \Rep_k(G)$, where $G/k$ is an affine group scheme, this is equivalent to saying that $G$ is connected.

\begin{lemma}
\label{lemma:characterization-of-cm-objects}
Let $\mathcal{T}$ be a neutral, connected Tannakian category over $k$ (where $k$ is assumed to have characteristic $0$). Then an object $X\in \mathcal{T}$ admits CM if and only if there exists a commutative, semisimple $k$-algebra $E$ of dimension $\mathrm{rk}(X)$ and an injection (of $k$-algebras)
$$
E\hookrightarrow \mathrm{End}_{\mathcal{T}}(X).
$$   
\end{lemma}
\begin{proof}
Assume that $X$ is a CM-object. Then, replacing $\mathcal{T}$ by $\langle X\rangle^\otimes$, we may assume that $\mathcal{T}\cong \Rep_k(T)$ for $T/k$ a torus (as $\mathcal{T}$ is neutral and connected). In other words, $X$ corresponds to a representation $V$ of the torus $T$. Decomposing $X$ into a sum of simple objects, we may assume that $X$ is simple. Over an algebraic closure $\overline{k}$ of $k$ the representation $X$ splits into a direct sum 
$$
X\otimes_{k}\overline{k}\cong X_{\chi_1}\oplus\ldots \oplus X_{\chi_n}
$$ 
of eigenspaces for distinct characters $\chi_i\colon T_{\overline{k}}\to \mathbb{G}_m$. As $X$ is simple all eigenspaces over $\overline{k}$ must be one-dimensional. In particular, we see that the endomorphism algebra $\mathrm{End}_{\mathcal{T}}(X)$ of $X$, the formation of which commutes with extensions of $k$, is a commutative semisimple algebra of dimension the rank of $X$.
Actually, $\mathrm{End}_{\mathcal{T}}(X)$ is a field in this case as $X$ is simple.

Conversely, assume that there exists an injection $E\hookrightarrow \mathrm{End}_{\mathcal{T}}(X)$ of a commutative, semisimple $k$-algebra $E$ such that $\mathrm{dim}_kE=\mathrm{rk}(X)$. We may pass to a finite extension of $k$ (cf.\ the remarks after \Cref{definition:cm-object}) and assume that $E\cong\prod\limits_{i=1}^nk$ is isomorphic to copies of $k$. The idempotents in $E$ define a decomposition $X\cong \bigoplus\limits_{i=1}^n X_i$ with $X_i$ of rank $1$ over $k$. Hence, replacing $X$ by one of the $X_i$, we may assume that $E=k$ and $\mathrm{rk}(X)=1$ (cf.\ \Cref{lemma:subquotients-of-cm-objects-again-cm}).
But as $\mathcal{T}$ is neutral, i.e., $\mathcal{T}\cong \Rep_k(G)$ for some affine group scheme $G$, a representation $\rho\colon G\to \mathbb{G}_m$ of rank $1$ defines a CM-object $X\in \mathcal{T}$ because the automorphism group scheme of the canonical fiber functor on $\langle X\rangle^\otimes$ is given by the image $\rho(G)$ of $\rho$ which is multiplicative.  
\end{proof}

In general it can happen that for a simple CM-object $X$ in a (non-neutral) Tannakian category in characteristic $0$ the endomorphisms
$$
\mathrm{End}_{\mathcal{T}}(X)
$$  
do not form a field. For example, this happens if $\mathcal{T}$ is the category of isocrystals over $\overline{\F}_p$. 

We record the following terminology.
\begin{definition}
\label{definition: cm-by-some-algebra}
If $X\in \mathcal{T}$ is an object in the Tannakian category $\mathcal{T}$ and $\iota\colon E\hookrightarrow \mathrm{End}_{\mathcal{T}}(X)$ an injection of a commutative semisimple $k$-algebra $E$ such that $\mathrm{rk}(X)=\mathrm{dim}_kE$, we say that $X$ admits CM by $E$.   
\end{definition}

Now we assume that the Tannakian category $\mathcal{T}$ is neutral (of characteristic $0$) and we fix a fiber functor
$$
\omega_0\colon\mathcal{T}\to \mathrm{Vec}_{k}. 
$$
Moreover, we assume that for some field extension $C/k$ the base extension
$$
\omega_0\otimes_k C\colon \mathcal{T}\to \mathrm{Vec}_C
$$
is equipped with a filtration, i.e., we fix a filtered fiber functor
$$
\omega\colon \mathcal{T}\to \mathrm{FilVec}_C
$$
such that $\omega\cong \omega_0\otimes_k C$.
We assume furthermore that $C$ is algebraically closed.
Of course this situation models the case of rational Hodge structures. But it will apply as well to the case of (rigidified) Breuil-Kisin-Fargues modules up to isogeny (cf.\ \Cref{lemma: filtered_fiber_functor_on_rigidified_bkf_modules}).

In this situation we can define the type of a CM-object.

\begin{definition}
\label{definition:type-of-a-cm-object}
Let $E/k$ be a commutative, semisimple $k$-algebra and let $\mathcal{T}$ be a Tannakian category over $k$, equipped with fiber functors as above. For a pair $(X,\iota)$ with $X\in \mathcal{T}$ an object, necessarily admitting CM, and an injection $\iota\colon E\to \mathrm{End}_{\mathcal{T}}(X)$ such that $n:=\mathrm{rk}(X)=\mathrm{dim}_kE$ we define the type 
$$
\Phi\colon \mathrm{Hom}_k(E,C)\to \Z
$$
of $(X,\iota)$ to be the unique function such that for every $i\in \Z$
$$
\mathrm{gr}^i(\omega(X))\cong \prod_{\tau\in \Phi^{-1}(i)}C_\tau
$$
as a representation of 
$$
E\otimes_{k}C\cong \prod\limits_{\tau\in \mathrm{Hom}_k(E,C)} C_\tau.
$$
\end{definition}

If $E/k$ is a commutative semisimple algebra and $\Phi\colon \mathrm{Hom}_{k}(E,C)\to \Z$ a function, then we call $(E,\Phi)$ a CM-type (over $k$).
Note that
$$
\mathrm{Hom}_{k}(E,C)\cong \mathrm{Hom}_k(E,\overline{k})
$$
where $\overline{k}\subseteq C$ denotes the algebraic closure of $k$ in $C$. In particular, the Galois group $\mathrm{Gal}(\overline{k}/k)$ acts naturally on the set
$$
\mathrm{Hom}_{k}(E,C).
$$

\begin{definition}
\label{definition-reflex-field}
Let $(E,\Phi)$ be a CM-type. Then the reflex field $E_\Phi\subseteq C$ of $(E,\Phi)$ is defined to be the fixed field $E_\Phi\subseteq \overline{k}$ of the stabilizer of $\Phi\colon \mathrm{Hom}_k(E,C)\to\Z$.
\end{definition}

If we write $E=\prod E_i$ as a product of fields and thus accordingly $\Phi=\coprod \Phi_i$ for functions $\Phi\colon \mathrm{Hom}_{k}(E_i,C)\to \Z$, then the reflex field $E_\Phi$ of $(E,\Phi)$ is the composite (in $\overline{k}$) of the reflex fields $E_{\Phi_i}$ of the CM-types $(E_i,\Phi_i)$.

Let $E$ be a commutative semisimple algebra over $k$ and set $T:=\mathrm{Res}_{E/k}\mathbb{G}_m$ to be the Weil restriction of the torus $\mathbb{G}_m$ over $E$.
Then the group
$$
X_\ast(T):=\mathrm{Hom}_{\overline{k}}(\mathbb{G}_{m,\overline{k}},T_{\overline{k}})
$$
of cocharacters of $T$ is isomorphic (as a Galois module) to the module
$$
\{\Phi\colon \mathrm{Hom}_k(E,C)\to \Z\}
$$
of types $\Phi$.
Indeed, given a type $\Phi\colon \mathrm{Hom}_k(E,C)\to \Z$ we get the associated cocharacter
$$
\mu_\Phi\colon \mathbb{G}_{m,\overline{k}}\to T_{\overline{k}}\cong\prod\limits_{\tau\in \mathrm{Hom}_k(E,C)}\mathbb{G}_{m,\overline{k}},\ t\mapsto (t^{\Phi(\tau)})_\tau
$$
of $T$ over $\overline{k}$.
By definition the reflex field $E_\Phi$ of the type $(E,\Phi)$ is the minimal subfield of $\overline{k}$ over which the cocharacter $\mu_\Phi$ is defined. In particular, we obtain a cocharacter
$$
\mu_\Phi\colon \mathbb{G}_{m,E_\Phi}\to T_{E_\Phi}.
$$
In the end, we obtain the reflex norm of $(E,\Phi)$ as the composition
$$
r_\Phi\colon \mathrm{Res}_{E_\Phi/k}(\mathbb{G}_m)\xrightarrow{\mathrm{Res}(\mu_\Phi)} \mathrm{Res}_{E_\Phi/k}(T_{E_\Phi})\xrightarrow{N} T
$$
where the second morphism denotes the natural norm map.

Let $L/k$ be a finite field extension contained in $\overline{k}$. Then we denote by
$$
L^\ast:=\mathrm{Res}_{L/k}(\mathbb{G}_m)
$$
the Weil restriction of the multiplicative group $\mathbb{G}_m$ over $L$ to $k$.
The character group of the torus $L^\ast$ is naturally isomorphic to the Galois module
$$
\Z[\mathrm{Hom}_{k}(L,\overline{k})].
$$
Concretely, let $\tau\colon L\to \overline{k}$ be an embedding (over $k$) and let $R/\overline{k}$ be a $\overline{k}$-algebra. Then the character $\chi_\tau\colon L^\ast_{\overline{k}}\to \mathbb{G}_{m,\overline{k}}$ is given on $R$-points by
$$
L^\ast_{\overline{k}}(R)=(R\otimes_k L)^\times\xrightarrow{\mathrm{Id}_R\otimes \tau}(R\otimes_k\overline{k})^\times\xrightarrow{\text{mult.}}R^\times
$$
where the right arrow denotes the multiplication $R\otimes_k\overline{k}\to R$ of the $\overline{k}$-algebra $R$.

Let $k\subseteq L_1\subseteq L_2\subseteq \overline{k}$ be a tower of field extensions.
Then there are natural norm maps
$$
N_{L_2/L_1}\colon L_2^\ast\to L_1^\ast
$$
which on $R$-valued points for a $k$-algebra $R$ are given by
$$
\begin{matrix}
L_2^\ast(R)=(R\otimes_k L_2)^\times&\to & (R\otimes_k L_1)^\times=L_1^\ast(R)\\
x& \mapsto & \mathrm{det}_{R\otimes_k L_1}(x|R\otimes_k L_2)  
\end{matrix}
$$
where the determinant is taken of the multiplication by $x$ on the finite free $R\otimes_k L_1$-module $R\otimes_k L_2$.
We define the pro-torus
$$
D_k:=\varprojlim\limits_{L\subseteq \overline{k}}L^\ast
$$
where the transition maps are given by the norms $N_{L_2/L_1}$ for $L_1\subseteq L_2$. In particular, the group of characters of $D_k$ is given by
$$
\varinjlim\limits_{L\subseteq \overline{k}} \Z[\mathrm{Hom}_k(L,\overline{k})]
$$
where for $L_1\subseteq L_2$ the transition morphism
$$
\Z[\mathrm{Hom}_k(L_1,\overline{k})]\to \Z[\mathrm{Hom}_k(L_2,\overline{k})]
$$
send $\tau\colon L_1\to \overline{k}$ to the sum
$$
\sum\limits_{\tau^\prime\colon L_2\to \overline{k} \atop \tau^\prime_{|L_1}=\tau}\tau^\prime.
$$
In particular, we see that the norm morphisms of tori
$$
N_{L_2/L_1}\colon L^\ast_2\to L^\ast_1
$$
are surjective (as fppf-sheaves).

\begin{lemma}
\label{lemma-character-group-of-pro-torus}
Let $G:=\mathrm{Gal}(\overline{k}/k)$.
The character group $X^\ast(D_k)$ of the pro-torus $D_k$ is canonically isomorphic to the coinduced module
$$
\mathrm{Coind}^{G}\Z:=\{f\colon G\to \Z\ |\ f \text{ has open stabilizer in } G \}
$$  
\end{lemma}
\begin{proof}
This is a general statemenet about discrete $G$-modules for $G$ a profinite group. Namely, for $H\subseteq G$ open,
$$
\mathrm{Coind}^{G}_H\Z:=\{f\colon G\to \Z\ |\ f \text{ is constant on } H-\text{cosets} \}
$$
is isomorphic to the free abelian group $\Z[G/H]$ via
$$
f\mapsto \sum\limits_{g\in G/H}f(g)g
$$
and
$$
\mathrm{Coind}^G\Z\cong \varinjlim\limits_{H\subseteq G \text{ open}}\mathrm{Coind}^G_H\Z
$$
by the continuity requirement in the definition of $\mathrm{Coind}^G\Z$.
Moreover, one checks that the transition maps agree under the isomorphism $\mathrm{Coind}^G_H\Z\cong \Z[G/H]$.
\end{proof}

Recall that we have fixed a fiber functor 
$$
\omega_0\colon \mathcal{T}\to \mathrm{Vec}_k
$$
such that its base extension 
$$
\omega:=\omega_0\otimes_k C\colon \mathcal{T}\to \mathrm{Vec}_C
$$
underlies a filtered fiber functor.

\begin{lemma}
\label{lemma-serre-group-via-reflex-fields}
There is a natural tensor functor
$$
r\colon \mathcal{T}_{\mathrm{CM}}\to \mathrm{Rep}_{k}(D_k),
$$
which can be explicitly described as follows:
Let $X\in \mathcal{T}_{\mathrm{CM}}$ be a CM-object with CM-type $(E,\Phi)$ and let $E_\Phi\subseteq \overline{k}$ be the reflex field of $(E,\Phi)$. Then the reflex norm
$$
r_\Phi\colon E_\Phi^\ast\to E^\ast
$$ 
defines an action of the torus $E_\Phi^\ast$ on $\omega_0(X)$, which defines the action of $D_k$ on $\omega_0(X)$ as $D_k\twoheadrightarrow E_{\Phi}^\ast$ surjects onto $E_\Phi^\ast$.
\end{lemma}
\begin{proof}
Let $G:=\mathrm{Aut}^\otimes(\omega_0)$ be the automorphism group of the fiber functor $\omega_0$. Then the graded fiber functor
$$
\mathrm{gr}(\omega)\colon \mathcal{T}\to \mathrm{GrVec}_C,\ X\mapsto \bigoplus\limits_{n\in\Z} \mathrm{gr}^n(\omega(X))
$$ 
defines a cocharacter
$$
\mu\colon \mathbb{G}_{m,\overline{k}}\to G_{\overline{k}}
$$
over $\overline{k}$ (this is spelled out more generally in \cite[Construction 3.4.]{ziegler_graded_and_filtered_fiber_functors}).
Let $T$ be the maximal torus quotient of $G$, i.e., 
$$
T=\mathrm{Aut}^\otimes({\omega_{0}}_{|\mathcal{T}_{\mathrm{CM}}})
$$
is the Tannakian fundamental group of the Tannakian category $\mathcal{T}_{\mathrm{CM}}\subseteq \mathcal{T}$ of CM-objects in $\mathcal{T}$. Composition with the canonical morphism $G\to T$ yields the cocharacter
$$
\mu\colon \mathbb{G}_{m,\overline{k}}\to T_{\overline{k}}
$$
of $T$ (which we denote by the same letter).
On character groups this corresponds to a morphism
$$
\mu^\ast\colon X^\ast(T)\to \Z
$$
of abelian groups. By the universal property of coinduction this defines a Galois-equivariant morphism
$$
\widetilde{\mu^\ast}\colon X^\ast(T)\to \mathrm{Coind}^{\mathrm{Gal}(\overline{k}/k)}\Z,\ \chi\mapsto (g\mapsto \chi\circ \alpha_{g^{-1}}\circ \mu)
$$
where for $g\in \mathrm{Gal}(\overline{k}/k)$ the morphism $\alpha_g\colon T_{\overline{k}}\to T_{\overline{k}}$ denotes the action of $g$ on $T_{\overline{k}}$.
In particular, we obtain (using \Cref{lemma-character-group-of-pro-torus}) the canonical tensor functor 
$$
r\colon \mathcal{T}_{\mathrm{CM}}\colon \mathrm{Rep}_k(T)\to \mathrm{Rep}_k(D_k)
$$
as the morphism induced by the morphism (over $k$)
$$
\widetilde{\mu^\ast}\colon D_k\to T
$$
associated with $\widetilde{\mu^\ast}\colon X^\ast(T)\to \mathrm{Coind}^{\mathrm{Gal}(\overline{k}/k)}\Z\cong X^\ast(D_k)$. We now identify the functor $r$ with the construction given in the lemma. Let $X\in \mathcal{T}$ be a CM-object with CM-type $(E,\Phi)$. Without loss of generality let $X$ be simple. By naturality of the construction in $\mathcal{T}$ we may further assume that $\mathcal{T}=\langle X\rangle^\otimes$. Then
$$
E=\mathrm{End}_{\mathcal{T}}(X)
$$
(cf.\ the proof of \Cref{lemma:characterization-of-cm-objects})
and we get
$$
T=\Aut^\otimes(\omega_0)\cong E^\ast
$$
with $T$ acting on $X$ via the action of $E^\ast$.
Indeed, $T$ equals the centralizer $C_{\mathrm{End}_{k}(\omega_0(X))}(E^\ast)$ of $E^\ast$ in the endomorphisms of $\omega_0(X)$, which itself is $E^\ast$. The containment 
$$
T\subseteq C_{\mathrm{End}_{k}(\omega_0(X))}(E^\ast)
$$ 
is clear and equality follows after base change to $\overline{k}$ where $X$ splits into $1$-dimensional representations with distinct characters. Let $E_\Phi$ be the reflex field of $(E,\Phi)$ (which is given by the stabilizer of $\mu$), and set $H:=\mathrm{Gal}(\overline{k}/E_\Phi)\subseteq G:=\mathrm{Gal}(\overline{k}/k)$.
In particular, 
$$
\mu\colon \mathbb{G}_{m,\overline{k}}\to T_{\overline{k}}=E^\ast_{\overline{k}}
$$ 
is already defined over $E_\Phi$.
We have to show that the diagram
$$
\xymatrix{
D_k\ar@{->>}[r]\ar[rrd]^{\widetilde{\mu^\ast}} & E^\ast_\Phi \ar[r]^-{\mathrm{Res}_{E_\Phi/k}(\mu)} & \mathrm{Res}_{E_\Phi/k}(T_{E_\Phi}) \ar[d]^{\text{Norm}} \\
& & T=E^\ast
}
$$
commutes. On character groups it gives rise to the diagram
$$
\xymatrix{
\mathrm{Coind}^G\Z & \mathrm{Coind}^G_H(\Z)\ar[l] & \mathrm{Coind}^G_H(X^\ast(T))\ar[l]_{\mu^\ast} \\
& & X^\ast(T)\ar[llu]^{\widetilde{\mu^\ast}}\ar[u]
}
$$
where the arrows without a label are the canonical ones. Here we note that $\mu^\ast\colon X^\ast(T)\to \Z$ defines an $H$-equivariant map by the definition of the reflex field (note that the $G$-module of maps $\Psi\colon \mathrm{Hom}_k(E,C)\to \Z$ is canonically isomorphic to the cocharacter group $X_\ast(T)$ of $T$ because $T=E^\ast$). Using the universal property of coinduction it suffices to check that the diagram commutes after composition with the canonical map
$$
\mathrm{Coind}^G(\Z)\to \Z,\ f\mapsto f(1).
$$ 
But then the commutativity follows directly from the definitions.
\end{proof}

In the case of rational Hodge structures the above functor yields the classical (connected) Serre group (cf.\ \cite{fargues_abelian_motives}) as a quotient of the pro-torus $D_\Q$.
In the case of Breuil-Kisin-Fargues modules we will show that the functor $r$ in \Cref{lemma-serre-group-via-reflex-fields} is an equivalence using the following criterion.
We define the full Tannakian subcategory
$$
\mathcal{T}_0\subseteq \mathcal{T}
$$
consisting of all objects $X\in \mathcal{T}$ such that
$$
\mathrm{gr}^i(\omega(X))=0
$$
for $i\neq 0$, where
$$
\omega\colon \mathcal{T}\to \mathrm{FilVec}_C
$$
is our fixed filtered fiber functor.
Moreover, for a finite field extension $E/k$ and an embedding $\tau\colon E\to C$ we denote by 
$$
\Phi_\tau\colon \mathrm{Hom}_k(E,C)\to \Z
$$ 
the map sending $\tau$ to $1$ and all $\tau^\prime\neq \tau$ to $0$. Clearly, the reflex field of $(E,\Phi_\tau)$ is $\tau(E)\subseteq \overline{k}\subseteq C$. 

\begin{corollary}
\label{corollary:reflex-norm-defines-equivalence}
Assume that the canonical functor $\mathrm{Vec}_k\to \mathcal{T}_0$ is an equivalence. Then the functor
$$
r\colon \mathcal{T}_{\mathrm{CM}}\to \mathrm{Rep}_kD_k
$$  
from \Cref{lemma-serre-group-via-reflex-fields} is fully faithful. Assume moreover that for every finite field extension $E/k$ and every embedding $\tau\colon E\to C$ there exists an object $X\in \mathcal{T}$ with CM by $(E,\Phi_\tau)$. Then the functor $r$ is essentially surjective. If this is the case, then furthermore $\mathcal{T}_{\mathrm{CM}}$ is generated by objects $X\in \mathcal{T}$ with CM by $(E,\Phi_\tau)$ for $E/k$ a finite field extension and $\tau\colon E\to C$ an embedding.
\end{corollary}
\begin{proof}
We can assume $\mathcal{T}_{\mathrm{CM}}=\mathcal{T}$.
Define
$$
T:=\mathrm{Aut}^\otimes(\omega_0)
$$
as the torus representing the tensor automorphisms of the fixed fiber functor
$$
\omega_0\colon \mathcal{T}\to \mathrm{Vec}_k.
$$
Then $\mathcal{T}\cong \mathrm{Rep}_k(T)$ and the functor $r$ in \Cref{lemma-serre-group-via-reflex-fields} is given by the morphism
$$
\widetilde{\mu^\ast}\colon D_k\to T
$$
(in the notation of \Cref{lemma-serre-group-via-reflex-fields}) corresponding to the morphism
$$
\widetilde{\mu^\ast}\colon X^\ast(T)\to X^\ast(D_k)\cong \mathrm{Coind}^G(\Z)
$$
on character groups. Let
$$
\mu\colon \mathbb{G}_{m,\overline{k}}\to T_{\overline{k}}
$$
be the cocharacter induced by the (graded fiber functor associated with the) filtered fiber functor $\omega$ and let
$$
\mu^\ast\colon X^\ast(T)\to \Z
$$
be the induced morphism on characters.
The assumption that the subcategory $\mathcal{T}_0$ is equivalent to the category $\mathrm{Vec}_k$ of $k$-vector spaces implies that the kernel $\mathrm{Ker}(\mu^\ast)$ contains no non-trivial orbit under the Galois group 
$$
G:=\mathrm{Gal}(\overline{k}/k).
$$ 
Indeed, if the orbit under some non-trivial $\chi\in X^\ast(T)$ under $G$ lies in $\mathrm{ker}(\mu^\ast)$, then the sum
$$
\sum\limits_{g\in G/\mathrm{Stab}_{G}(\chi)}g\chi
$$
defines a representation $X\in \mathrm{Rep}_k(T)\cong \mathcal{T}$ of $T$ over $k$ such that $\mathrm{gr}^i(\omega(X))=0$ for $i\neq 0$. In other words, $X\in \mathcal{T}_0$. But this is a contradiction as $T$ acts by assumption trivially on any object in $\mathcal{T}_0$. Thus the kernel $\mathrm{Ker}(\mu^\ast)$ does not contain a non-trivial $G$-orbit.
The kernel $\mathrm{Ker}(\mu^\ast)$ contains the kernel $\mathrm{Ker}(\widetilde{\mu^\ast})$. Moreover, $\mathrm{Ker}(\widetilde{\mu^\ast})$ is stable under $G$ and hence trivial as it does not contain a non-trivial $G$-orbit, as well. Hence, the morphism
$$
X^\ast(T)\to \mathrm{Coind}^G(\Z)
$$
is injective and therefore the morphism $D_k\to T$ surjective (as a morphism of fppf-sheaves). This implies that the functor $r\colon \mathcal{T}\to \mathrm{Rep}_k(D_k)$ is fully faithful. The character group 
$$
\mathrm{Coind}^G(\Z)\cong \varinjlim\limits_{L/k}\Z[\mathrm{Hom}_k(L,C)]
$$ 
of $D_k$ is generated (as a Galois module) by embeddings $\tau\colon E\to C$ for finite field extensions $E/k$ (more precisely, by the canonical embedding $L\subseteq C$ for $L:=\tau(E)$). Let $X\in \mathcal{T}_{\mathrm{CM}}$ be an object with CM by $(E,\Phi_\tau)$. Then the reflex field $E_{\Phi_\tau}$ of the CM-type $(E,\Phi_\tau)$ is given by
$$
E_{\Phi_\tau}:=\tau(E)\subseteq \overline{k}\subseteq C.
$$
Moreover, the reflex norm
$$
r_{\Phi_\tau}\colon E_{\Phi_\tau}^\ast\xrightarrow{\tau^{-1}}E^\ast
$$
is induced by the inverse of $\tau\colon E\cong E_{\Phi_\tau}$. Indeed, let $H=\mathrm{Gal}(\overline{k}/E_{\Phi_\tau})\subseteq G$ be the stabilizer of $\Phi_\tau$.
Then the reflex norm $E_{\Phi_\tau}^\ast\to E^\ast$ corresponds to the $G$-invariant morphism
$$
\Z[\mathrm{Hom}_k(E,C)]\to \mathrm{Coind}^G_H(\Z[\mathrm{Hom}_k(E,C)])\xrightarrow{\Phi_\tau} \mathrm{Coind}^G_H\Z\cong \Z[\mathrm{Hom}_k(E_{\Phi_\tau},C)]
$$
on character lattices and $\tau\in \mathrm{Hom}_k(E,C)$ is sent to the canonical inclusion 
$$
E_{\Phi_\tau}\subseteq C.
$$
In particular, we can conclude that $X$ must be simple. Namely, its image 
$$
r(X)\in \mathrm{Rep}_k(D_k)
$$
is given by the $k$-vector space $E$ with $D_k$ through its quotient $E_{\Phi_\tau}^\ast$ acting via $\tau^{-1}\colon E_{\Phi_\tau}^\ast\to E^\ast$.
In particular, $r(X)$ and hence $X$ are simple.
As in the proof of \Cref{lemma-serre-group-via-reflex-fields} we see that $T\twoheadrightarrow E^\ast$ surjects onto the torus $E^\ast$ and thus the functor
$$
\mathrm{Rep}_k(E^\ast)\to \mathcal{T}
$$
is fully faithful. 
As the above morphism
$$
\Z[\mathrm{Hom}_k(E,C)]\cong \Z[\mathrm{Hom}_k(E_{\Phi_\tau},C)]
$$  
is an isomorphism we see that 
$$
\Z[\mathrm{Hom}_k(E_{\Phi_\tau},C)]\subseteq \mathrm{Coind}^G\Z
$$
lies in the image of $X^\ast(T)$. As $E/k$ was arbitrary, we can conclude that
$$
X^\ast(T)\cong \mathrm{Coind}^G\Z
$$
which implies
$$
\mathcal{T}\cong \mathrm{Rep}_k(D_k).
$$
For the last statement it suffices to note, as was mentioned above, that the character group
$$
\mathrm{Coind}^G(\Z)\cong\varinjlim\limits_{\overline{k}/L/k} \Z[\mathrm{Hom}_k(L,C)]
$$
is generated by embeddings $\tau\colon L\to C$ for $L/k$ finite.
\end{proof}

\section{Breuil-Kisin-Fargues modules up to isogeny}
\label{section: breuil_kisin_fargues_modules_up_to_isogeny}

We try to follow closely the notations in \cite{scholze_weinstein_lecture_notes_on_p_adic_geometry} and \cite{bhatt_morrow_scholze_integral_p_adic_hodge_theory}.
Fix a prime $p$.
Let $C/{\Q_p}$ be a complete, non-archimedean, algebraically closed extension of $\Q_p$, e.g., $C=\mathbb{C}_p=\widehat{\overline{\Q_p}}$ the completion of an algebraic closure of $\Q_p$.
Let $\mathcal{O}_C\subseteq C$ be the ring of integers in $C$ and let
$$
C^\flat:=\varprojlim\limits_{x\mapsto x^p}C
$$  
be the tilt of $C$ with its ring of integers
$$
\mathcal{O}_{C^\flat}=\mathcal{O}_C^\flat=\varprojlim\limits_{x\mapsto x^p}\mathcal{O}_C\cong \varprojlim\limits_{x\mapsto x^p}\mathcal{O}_C/p.
$$
We denote by
$$
(-)^\#\colon C^\flat\to C,\ (x_0,x_1,\ldots)\mapsto x_0
$$
resp.\
$$
(-)^\#\colon \mathcal{O}_{C^\flat}\to \mathcal{O}_C,\ (x_0,x_1,\ldots)\mapsto x_0
$$
the canonical projections.
Let 
$$
\epsilon:=(1,\zeta_p,\zeta_{p^2},\ldots)\in \mathcal{O}_{C^\flat}
$$ 
be a fixed system of $p$-power roots of unity in $C^\flat$ (the results of this paper will be independent of the choice of $\epsilon$).
Define
$$
A_\inf:=W(\mathcal{O}_{C^\flat})
$$
as the ring of Witt vectors for the (perfect) ring $\mathcal{O}_{C^\flat}$.
We denote by
$$
[-]\colon \mathcal{O}_{C^\flat}\to W(\mathcal{O}_{C^\flat})
$$
the canonical Teichm\"uller lift.
When dealing with $A_\inf$ as a topological ring we always equip it with the $(p,[\varpi])$-adic topology for a pseudo-uniformizer $\varpi\in C^\flat$ (this topology is independent of the choice of $\varpi$). 
Let
$$
\varphi\colon A_\inf\to A_\inf,\ \sum\limits_{i=0}^\infty [x_i]p^i\mapsto \sum\limits_{i=0}^\infty [x_i^p]p^i 
$$
be the Frobenius on $A_\inf$.
We define the elements
$$
\begin{matrix}
\mu:=[\epsilon]-1\\
\xi:=\frac{[\epsilon]-1}{[\epsilon^{1/p}]-1}=\sum\limits_{i=0}^{p-1}[\epsilon^{i/{p-1}}]\\
\tilde{\xi}:=\varphi(\xi)=\frac{[\epsilon^p]-1}{[\epsilon]-1}
\end{matrix}
$$
of $A_\inf$.
Let
$$
\theta\colon A_\inf\to \mathcal{O}_C,\ \sum\limits_{i=0}^\infty [x_i]p^i\mapsto \sum\limits_{i=0}^\infty (x_i)^\#p^i
$$
be Fontaine's map $\theta$ and set
$$
\tilde{\theta}:=\theta\circ \varphi^{-1}.
$$
Then $\xi\in A_\inf$ is a generator of $\ker(\theta)$
and $\tilde{\xi}\in A_\inf$ a generator of $\ker(\tilde{\theta})$.
We define as usual
$$
A_\crys
$$
as the universal PD-thickening of $\theta$ and $B^+_\crys:=A_\crys[\frac{1}{p}]$.
Finally, Fontaine's period ring $B^+_\dR$ is defined to be the $\xi$-adic completion of $A_\inf[\frac{1}{p}]$. In particular, the map $\theta$ induces a (surjective) morphism
$$
\theta\colon B^+_{\dR}\to C.
$$
In fact, the ring $B^+_{\dR}$ is a complete discrete valuation ring with uniformiser $\xi$ and residue field $C$. Its fraction field $B_{\dR}:=B^+_{\dR}[\frac{1}{\xi}]$ is also called the field of $p$-adic periods.

Having fixed this notation we can now define the notion of a Breuil-Kisin-Fargues module which is a ``perfectoid'' analog of a Breuil-Kisin module.

\begin{definition}
\label{definition: bkf_module}
 A Breuil-Kisin-Fargues module is a finitely presented $A_\inf$-module $M$ with an isomorphism
$$
\varphi_M\colon (\varphi^\ast M)[\frac{1}{\tilde{\xi}}]=M\otimes_{A_\inf,\varphi}A_\inf[\frac{1}{\tilde{\xi}}]\xrightarrow{\cong} M[\frac{1}{\tilde{\xi}}]
$$ 
such that $M[\frac{1}{p}]$ is a finite projective $A_\inf[\frac{1}{\tilde{\xi}}]$-module.
\end{definition}

By \cite[Corollary 4.12.]{bhatt_morrow_scholze_integral_p_adic_hodge_theory} a finite projective $A_\inf[\frac{1}{p}]$-module is automatically free, hence in \Cref{definition: bkf_module} the condition ``finite projective'' can be replaced by ``finite free".
We denote by
$$
\mathrm{BKF}
$$
the category of Breuil-Kisin-Fargues modules (cf.\ \cite[Remark 4.25]{bhatt_morrow_scholze_integral_p_adic_hodge_theory}). It is (naturally) an exact tensor category. However, it is not abelian, only pseudo-abelian.

\begin{lemma}
\label{lemma:bkf-modules-are-pseudo-abelian}
Let $(M,\varphi_M)$ be a Breuil-Kisin-Fargues module and let $e\colon M\to M$ be an idempotent endomorphism of $(M,\varphi_M)$, i.e., $e$ commutes with $\varphi_M$. Then $\mathrm{ker}(e)$ and $\mathrm{coker}(e)$ with their induced Frobenii are Breuil-Kisin-Fargues modules. In particular, the category $\mathrm{BKF}$ of Breuil-Kisin-Fargues modules is pseudo-abelian.   
\end{lemma}
\begin{proof}
Clearly, the kernel and cokernel of $e$ admit Frobenii. As direct summands of finitely presented (resp.\ finite projective) modules are again finitely presented(resp.\ finite projective) $\mathrm{ker}(e)$ and $\mathrm{coker}(e)$ satisfy all conditions in \Cref{definition: bkf_module}, i.e., they are Breuil-Kisin-Fargues modules.
\end{proof}

We give an easy example of a Breuil-Kisin-Fargues modules.
\begin{example}
\label{example-bkf-modules-of-rank-1}
For $d\in \Z$ we set $A_\inf\{d\}:=\mu^{-d}A_\inf\otimes_{\Z_p}\Z_p(1)$ with Frobenius 
$$
\varphi_{A_\inf\{d\}}\colon \varphi^\ast(A_\inf\{d\})=\varphi(\mu)^{-d}A_\inf\otimes_{\Z_p}\Z_p(1)\to A_\inf\{d\},\ a\mapsto \tilde{\xi}^d a
$$

\end{example}

\begin{definition}
\label{definition: dual_of_free_bkf_module}
Let $(M,\varphi_M)$ be a Breuil-Kisin-Fargues module such that $M$ is finite projective (equivalently, free) over $A_\inf$. Then the dual $(M^\vee,\varphi_{M^\vee})$ of $(M,\varphi_M)$ is defined by
$$
M^\vee:=\mathrm{Hom}_{A_\inf}(M,A_\inf)
$$
with Frobenius
$$
\varphi_{M^\vee}\colon (\varphi^\ast M^\vee)[\frac{1}{\tilde{\xi}}]=(\varphi^\ast M)[\frac{1}{\tilde{\xi}}]^\vee\xrightarrow{(\varphi_M^{\vee})^{-1}}(M[\frac{1}{\tilde{\xi}}])^\vee=M^\vee[\frac{1}{\tilde{\xi}}].
$$
where we used the notation $(-)^\vee=\mathrm{Hom}_{A_{\inf}[\frac{1}{\tilde{\xi}}]}(-,A_{\inf}[\frac{1}{\tilde{\xi}}])$ again for duals of finite projective $A_{\inf}[\frac{1}{\tilde{\xi}}]$-modules.  
\end{definition}

We recall the following lemma about general $A_\inf$-modules.

\begin{lemma}
\label{lemma: structure_of_a_inf_modules}
Let $M$ be a finitely presented $A_\inf$-module such that $M[\frac{1}{p}]$ is finite projective (equivalently, free) over $A_\inf$. Then there is a functorial (in M) exact sequence
$$
0\to M_{\mathrm{tor}}\to M\to M_{\mathrm{free}}\to \overline{M}\to 0
$$
satisfying:
\begin{itemize}
\item[i)] $M_{\mathrm{tor}}$ is finitely presented and perfect as an $A_\inf$-module, and killed by $p^n$ for some $n\gg 0$.
\item[ii)] $M_{\mathrm{free}}$ is a finite free $A_\inf$-module.
\item[iii)] $\overline{M}$ is finitely presented and perfect as an $A_\inf$-module, and is supported at the closed point of $\Spec(A_\inf)$. 
\end{itemize}
\end{lemma}
\begin{proof}
See \cite[Proposition 4.13]{bhatt_morrow_scholze_integral_p_adic_hodge_theory}. 
\end{proof}

In particular, a finitely presented $A_\inf$-module $M$ such that $M[\frac{1}{p}]$ is finite free over $A_\inf[\frac{1}{p}]$ is perfect as an $A_\inf$-module, i.e., admits a finite projective resolution (cf.\ \cite[Lemma 4.9.]{bhatt_morrow_scholze_integral_p_adic_hodge_theory}). 

Moreover, if the $A_\inf$-module $M$ in \Cref{lemma: structure_of_a_inf_modules} is a Breuil-Kisin-Fargues module, i.e., equipped with a Frobenius 
$$
\varphi_M\colon \varphi^\ast(M)[\frac{1}{\tilde{\xi}}]\cong M[\frac{1}{\tilde{\xi}}],
$$ 
then the modules $M_{\mathrm{rm}}$, $M_{\mathrm{free}}$ and $\overline{M}$ carry a natural Frobenius as well and the exact sequence in \Cref{lemma: structure_of_a_inf_modules} is an exact sequence of Breuil-Kisin-Fargues modules (note that $M_{\mathrm{tor}}[\frac{1}{p}]=\overline{M}[\frac{1}{p}]=0$ is free). In fact, the existence of this Frobenius is clear for $M_{\mathrm{tor}}$. But $M_{\mathrm{free}}=H^0(\Spec(A_\inf),M/{M_{\mathrm{tor}}})$ (cf.\ the proof of \cite[Proposition 4.13]{bhatt_morrow_scholze_integral_p_adic_hodge_theory}) which yields the Frobenius on $M_{\mathrm{free}}$ (and thus $\overline{M}$ as well). 

We want to recall some equivalent descriptions of Breuil-Kisin-Fargues modules whose underlying $A_\inf$-module is finite free (cf.\ \cite[Proposition 20.1.1.]{scholze_weinstein_lecture_notes_on_p_adic_geometry}). For this we need to recall that associated to the fixed complete and algebraically closed non-archimedean field $C/\Q_p$ there is associated a scheme $X_{\mathrm{FF}}$ over $\Q_p$, ``the Fargues-Fontaine curve'', together with a distinguished point $\infty\in X_{\mathrm{FF}}$ whose completed stalk $\widehat{\mathcal{O}_{X_{\mathrm{FF}},\infty}}\cong B^+_\dR$ is isomorphic to Fontaine's period ring $B^+_{\dR}$ associated with $C$ (cf.\ \cite[Definition 13.5.3.]{scholze_weinstein_lecture_notes_on_p_adic_geometry}). We also recall that the adic Fargues-Fontaine curve $X_{\mathrm{FF}}^\ad=Y/\varphi^\Z$ is uniformized by an adic space $Y$ admitting an action of $\varphi$ (cf.\ \cite[Definition 13.5.1]{scholze_weinstein_lecture_notes_on_p_adic_geometry}).
The space $Y$ is defined to be 
$$
Y:=\mathrm{Spa}(A_\inf)\setminus V(p[\varpi])
$$
where $\varpi\in \mathcal{O}_{C^\flat}$ is a pseudo-uniformizer (cf.\ \cite[Proposition 13.1.1.]{scholze_weinstein_lecture_notes_on_p_adic_geometry} for a proof that this is actually an honest adic space, i.e., the structure presheaf is a sheaf).
Additionally to $Y$ the adic space $\mathrm{Spa}(A_\inf)\setminus V(p,[\varpi])$ contains two points, namely
\begin{itemize}
\item $x_\crys$ given as the image of the $p$-adic valuation on $W(k)$ along the canonical projection $A_\inf\to W(k)$ (this point is denoted $x_L$ in \cite{scholze_weinstein_lecture_notes_on_p_adic_geometry}).
\item $x_{\mathrm{\acute{e}tale}}$ given by the image of the valuation on $\mathcal{O}_{C^\flat}$ along the morphism $A_\inf\to O_{C^\flat}=A_\inf/p$ (this point is denoted $x_{C^\flat}$ in \cite{scholze_weinstein_lecture_notes_on_p_adic_geometry}). 
\end{itemize}

We set
$$
B:=H^0(Y,\mathcal{O}_{Y})
$$
and 
$$
B^+:=H^0(\Spa(A_\inf)\setminus{V(p)}, \mathcal{O}).
$$
There exists a natural $\varphi$-equivariant injection $B^+\to B^+_\crys$ making $B^+_\crys$ a $B^+$-algebra (cf.\ \cite[Proposition 1.10.12.]{fargues_fontaine_courbes_et_fibres_vectoriels_en_theorie_de_hodge_p_adique}).

\begin{theorem}
\label{theorem: vector_bundles_and_varphi_modules_over_b_+_crys}
The category $\mathrm{Bun}_{X_{\mathrm{FF}}}$ of vector bundles on the Fargues-Fontaine curve is naturally equivalent to the category of $\varphi$-modules over $B^+_\crys$, i.e., to finite projective $B^+_\crys$-modules $M$ together with an isomorphism $\varphi_M\colon \varphi^\ast M\cong M$.  
\end{theorem}
\begin{proof}
This is proven in \cite[Corollaire 11.1.14.]{fargues_fontaine_courbes_et_fibres_vectoriels_en_theorie_de_hodge_p_adique}. The natural functor inducing the equivalence is given by pulling back a vector bundle $\mathcal{F}\in \mathrm{Bun}_{X_{\mathrm{FF}}}\cong \Bun_{X^\ad_{\mathrm{FF}}}$ to $Y$, extending it along the point $x_{\crys}$ and tensoring the global sections of this extension (which form a $\varphi$-module over $B^+$) over $B^+$ with $B^+_{\crys}$.
\end{proof}

\begin{theorem}
\label{theorem: equivalent_descriptions_of_bkf_modules} 
The following categories are equivalent:
\begin{itemize}
\item[i)] Finite free Breuil-Kisin-Fargues modules $(M,\varphi_M)$.
\item[ii)] Pairs $(T,\Xi)$, where $T$ is a finite free $\Z_p$-module, and $\Xi\subseteq T\otimes_{\Z_p}B_\dR$ is a $B^+_\dR$-lattice.
\item[iii)] Quadruples $(\mathcal{F},\mathcal{F}^\prime,\beta,T)$, where $\mathcal{F}$ and $\mathcal{F}^\prime$ are vector bundles on the Fargues-Fontaine curve $X_{\mathrm{FF}}$, and $\beta\colon \mathcal{F}_{|X_{\mathrm{FF}\setminus{\infty}}}\xrightarrow{\simeq}\mathcal{F}^\prime_{|X_{\mathrm{FF}\setminus{\infty}}}$ is an isomorphism, $\mathcal{F}$ is trivial, and $T\subseteq H^0(X_{\mathrm{FF}},\mathcal{F})$ is a $\Z_p$-lattice.  
\end{itemize}
\end{theorem}
\begin{proof}
The proof is given in \cite[Proposition 20.1.1.]{scholze_weinstein_lecture_notes_on_p_adic_geometry}. We shortly give descriptions of the various functors involved.
Let $(M,\varphi_M)$ be a finite free Breuil-Kisin-Fargues module. Then the $\varphi_M$-invariants
$$
T:=(M\otimes_{A_\inf}W(C^\flat))^{\varphi_M=1}
$$
form a finite free $\Z_p$-module (by Artin-Schreier theory, cf.\ \cite[Theorem 12.3.4.]{scholze_weinstein_lecture_notes_on_p_adic_geometry}). Note that $\tilde{\xi}$ is invertible in $W(C^\flat)$, hence the base extension $M\otimes_{A_\inf}W(C^\flat)$ actually carries a Frobenius. Moreover, multiplication defines an isomorphism
$$
T\otimes_{\Z_p}W(C^\flat)\cong M\otimes_{A_\inf}W(C^\flat)
$$
under which $T\otimes_{\Z_p}A_{\inf}[\frac{1}{\mu}]\subseteq T\otimes_{\Z_p}W(C^\flat)$ is mapped to $M\otimes_{A_\inf}A_{\inf}[\frac{1}{\mu}]$ (cf.\ \cite[Lemma 4.26.]{bhatt_morrow_scholze_integral_p_adic_hodge_theory}).
Setting 
$$
\Xi:=M\otimes_{A_\inf}B^+_{\dR}\subseteq (M\otimes_{A_\inf}A_\inf[\frac{1}{\mu}])\otimes_{A_\inf[\frac{1}{\mu}]}B_{\dR}\cong T\otimes_{\Z_p}B_{\dR}
$$ defines the pair $(T,\Xi)$ in ii) associated with $(M,\varphi_M)$. If $\Xi\subseteq T\otimes_{\Z_p}B_{\dR}$ is a $B^+_\dR$-lattice, where $T$ is a finite free $\Z_p$-module, then (in order to pass to iii)) one can use this lattice to modify the trivial bundle $\mathcal{F}:=T\otimes_{\Z_p}\mathcal{O}_{X_{\mathrm{FF}}}$ at the distinguished point $\infty\in x_{\mathrm{FF}}$ on the Fargues-Fontaine curve to obtain a bundle $\mathcal{F}^\prime$ with an isomorphism $\mathcal{F}\dashrightarrow\mathcal{F}^\prime$ away from $\infty$ (recall that $\widehat{\mathcal{O}_{X_{\mathrm{FF},\infty}}}\cong B^+_{\dR}).$
Clearly, $T\subseteq H^0(X_{\mathrm{FF}},\mathcal{O}_{X_{\mathrm{FF}}})$ defines a $\Z_p$-lattice as $H^0(X_{\mathrm{FF}},\mathcal{O}_{X_\mathrm{FF}})\cong T\otimes_{\Z_p}\Q_p$. Moreover, this construction can be reversed giving the equivalence between the categories in ii) and iii). Thus in particular, the hard part of the theorem is the construction of a functor from the category in ii) (or iii)) to Breuil-Kisin-Fargues modules. We shortly describe how this is done. Let $(\mathcal{F},\mathcal{F}^\prime, \alpha,T)$ be given as in iii). Pulling back $\mathcal{F}$ and $\mathcal{F}^\prime$ to $Y$ defines two $\varphi$-equivariant vector bundles $\mathcal{E},\mathcal{E}^\prime$ on $Y$. By triviality of $\mathcal{F}$ the vector bundle $\mathcal{E}$ extends (non-canonically) to the point $x_{\mathrm{\acute{e}tale}}$. This extension is not unique, but depends on the choice of a $\Z_p$-lattice in $H^0(X,\mathcal{F})=H^0(Y,\mathcal{E})^{\varphi=1}$. In particular, the data of $T$ yiels a canonical extension of $\mathcal{E}$ to $x_{\mathrm{\acute{e}tale}}$. Moreover, every $\varphi$-equivariant vector bundle on $Y$, thus in particular $\mathcal{E}^\prime$, extends uniquely to the point $x_{\crys}$.
Using the modification $\mathcal{F}\dashrightarrow \mathcal{F}^\prime$ away from $\infty$ the bundles $\mathcal{E}$ and $\mathcal{E}^\prime$ are ($\varphi$-equivariantly) isomorphic when restricted to a small annulus omitting the points $\varphi^n(\infty),n\in \Z$. In particular, one obtains a vector bundle $\mathcal{E}^{\prime\prime}$ on $\Spa(A_\inf)\setminus V(p,[\varpi])$ which restricts to the extension of $\mathcal{E}^\prime$ near $x_\crys$ and to the extension of $\mathcal{E}$ near $x_{\mathrm{\acute{e}tale}}$. Taking global sections of $\mathcal{E}^{\prime\prime}$ defines the finite free $A_\inf$-module $M$ which is then moreover equipped with a Frobenius $\varphi_M$ away from $V(\tilde{\xi})$, i.e., a finite free Breuil-Kisin-Fargues module.
In particular, we see that the $\varphi$-module $M\otimes_{A_\inf}B^+_\crys$ over $B^+_\crys$ is canonically isomorphic to the $\varphi$-module over $B^+_\crys$ associated with $\mathcal{F}^\prime$ as in \Cref{theorem: vector_bundles_and_varphi_modules_over_b_+_crys}.   
\end{proof}

For example, the Breuil-Kisin-Fargues modules $A_\inf\{d\}$ is sent to the pair
$$
(\Z_p(d),\Xi:=\xi^{-d}(\Z_p(d)\otimes_{\Z_p}B^+_{\dR})\subseteq \Z_p(d)\otimes_{\Z_p}B_\dR)
$$
resp., for $d\geq 0$, to the modification
$$
0\to \Z_p(d)\otimes_{\Z_p} \mathcal{O}_{X_\mathrm{FF}}\xrightarrow{t^d} \mathcal{O}_{X_{\mathrm{FF}}}(d)\to t^{-d}\mathcal{O}_{X_{\mathrm{FF}},\infty}/\mathcal{O}_{X_{\mathrm{FF}},\infty}\to 0
$$
induced by the $d$-th power $t^d$ of Fontaine's $t=\mathrm{log}([\epsilon])\in H^0(X_{\mathrm{FF}},\mathcal{O}_{X_{\mathrm{FF}}}(1))=(B^+_\crys)^{\varphi=p}$.

We want to stress that the categories in \Cref{theorem: equivalent_descriptions_of_bkf_modules} are not equivalent \textit{as exact categories}, i.e., the equivalences (or their inverses) are not exact. In fact, in the construction of the finite free Breuil-Kisin-Fargues module one has to extend (by taking global sections) a vector bundle on the punctured $\Spec(A_\inf)$ with the closed point removed to the whole of $\Spec(A_\inf)$ and this is not an exact operation.

\begin{definition}
\label{definition: bkf_modules_up_to_isogeny}
We define the isogeny category of Breuil-Kisin-Fargues modules $\mathrm{BKF}^\circ$ as 
$$
\mathrm{BKF}^\circ:=\Q_p\otimes_{\Z_p}\mathrm{BKF},
$$
i.e., the category $\mathrm{BKF}^\circ$ has the same objects as the category $\mathrm{BKF}$, but for $M,N\in \mathrm{BKF}^\circ$ the space of homomorphisms is given by
$$
\mathrm{Hom}_{\mathrm{BKF}^\circ}(M,N):=\Q_p\otimes_{\Z_p}\mathrm{Hom}_{\mathrm{BKF}}(M,N).
$$    
\end{definition}

The category $\mathrm{BKF}^\circ$ is still not abelian (for the same reason as $\mathrm{BKF}$, cf.\ \cite[Remark 4.25]{bhatt_morrow_scholze_integral_p_adic_hodge_theory}), but rigid as we now show.

\begin{lemma}
\label{lemma: bkf_modules_rigid_exact_tensor_category}
The isogeny category $\mathrm{BKF}^\circ$ of Breuil-Kisin-Fargues modules is a rigid, exact tensor category.  
\end{lemma}
\begin{proof}
It suffices to prove that every object admits a dual because the tensor product and the exact structure are inherited from the category $\mathrm{BKF}$. Let $(M,\varphi_M)\in \mathrm{BKF}$ be a Breuil-Kisin-Kisin-Fargues module. We claim that, in the isogeny category, $(M,\varphi_M)$ is isomorphic to a Breuil-Kisin-Fargues module $(N,\varphi_N)$ such that $N$ is finite free over $A_\inf$. Let $M_{\mathrm{tor}}\subseteq M$ be the torsion submodule. Then $\varphi^\ast(M_{\mathrm{tor}})$ resp.\ $\varphi^\ast(M_{\mathrm{tor}})[\frac{1}{\tilde{\xi}}]$ is the torsion submodule of $\varphi^\ast(M)$ resp.\ $\varphi^\ast(M)[\frac{1}{\tilde{\xi}}]$. Therefore it is stable by $\varphi_M$. By lemma \Cref{lemma: structure_of_a_inf_modules} the torsion submodule $M_{\mathrm{tor}}$ is actually killed by $p^n$ for $n\gg 0$, and thus zero in the isogeny category. Hence, we may assume that $M$ is actually torsion-free. Set $N:=M_{\mathrm{free}}$ as in \Cref{lemma: structure_of_a_inf_modules}, which is finite free over $A_\inf$. Then $\varphi^\ast(N)[\frac{1}{\tilde{\xi}}]=\varphi^\ast(M)[\frac{1}{\tilde{\xi}}]$ and $N[\frac{1}{\tilde{\xi}}]=M[\frac{1}{\tilde{\xi}}]$ because the cokernel of $M\to N$ is killed by $\xi$resp.\  $\tilde{\xi}$.
In particular, the Frobenius $\varphi_M$ of $M$ defines a Breuil-Kisin-Fargues module $(N,\varphi_M)$ and we are finally allowed to replace $M$ by $N$, as $(M,\varphi_M)$ and $(N,\varphi_N)$ are isomorphic in the isogeny category. This finishes the proof as finite free Breuil-Kisin-Fargues modules clearly admit duals.  
\end{proof}

We remark that the proof of \Cref{lemma: bkf_modules_rigid_exact_tensor_category} actually shows that for a Breuil-Kisin-Fargues module $(M,\varphi_M)$ the exact sequence
$$
0\to M_{\mathrm{tor}}\to M\to M_{\mathrm{free}}\to \overline{M}\to 0
$$
of \Cref{lemma: structure_of_a_inf_modules} can be enhanced to an exact sequence of Breuil-Kisin-Fargues modules, i.e., there are compatible Frobenii on the modules involved.

In terms of pairs of a $\Z_p$-lattice and a $B^+_{\dR}$-lattice the category $\mathrm{BKF}^\circ$ of Breuil-Kisin-Fargues modules up to isogeny is equivalent to the category of pairs $(V,\Xi)$ where $V$ is a finite dimensional $\Q_p$-vector space and $\Xi\subseteq V\otimes_{\Q_p}B_{\dR}$ a $B^+_{\dR}$-lattice. Again, we can see that this category is not abelian: The canonical morphism
$$
(\Q_p,\xi B^+_{\dR})\to (\Q_p,B^+_\dR)
$$
does not admit a cokernel. However, we see that it admits kernels. But we caution the reader that it is not clear how to describe the kernel of a morphism $f\colon (M,\varphi_M)\to (N,\varphi_N)$ of Breuil-Kisin-Fargues modules up to isogeny, because it is not clear whether the kernel $K:=\mathrm{Ker}(M\xrightarrow{f}N)$ is finitely presented over $A_\inf$ or satisfies the property that $K[\frac{1}{p}]$ is finite projective over $A_\inf[\frac{1}{p}]$ (cf.\ the proof of \Cref{theorem: rigidified_bkf_modules_form_tannakian_category}).

We are seeking for a Tannakian, and thus in particular abelian, category of Breuil-Kisin-Fargues-modules up to isogeny. This is possible after adding a rigidification after base change to $B^+_{\crys}=A_\crys[\frac{1}{p}]$ (cf.\ \cite[Lemma 4.27.]{bhatt_morrow_scholze_integral_p_adic_hodge_theory}).
Let us recall \cite[Lemma 4.27.]{bhatt_morrow_scholze_integral_p_adic_hodge_theory}.

\begin{lemma}
\label{lemma: existence_of_rigidifications}
Let $(M,\varphi_M)$ be a Breuil-Kisin-Fargues module. Then $\overline{M}=M\otimes_{A_\inf}W(k)$ is a finitely generated $W(k)$-module equipped with a Frobenius automorphism after inverting $p$. Fix a section $k\to \mathcal{O}_C/p$, which induces a section $W(k)\to A_\inf$. Then there is a (noncanonical) $\varphi$-equivariant isomorphism
$$
M\otimes_{A_\inf}B^+_{\crys}\cong \overline{M}\otimes_{W(k)}B^+_\crys
$$ 
reducing to the identity over $W(k)[\frac{1}{p}]$.
\end{lemma}
\begin{proof}
See \cite[Lemma 4.27.]{bhatt_morrow_scholze_integral_p_adic_hodge_theory} which refers to \cite[Corollaire 11.1.14]{fargues_fontaine_courbes_et_fibres_vectoriels_en_theorie_de_hodge_p_adique}.  
\end{proof}

We remark that $\tilde{\xi}$ is invertible in $B^+_\crys$ (but $\xi$ is not), hence the Frobenius $\varphi_M\colon \varphi^{\ast}M[\frac{1}{\tilde{\xi}}]\xrightarrow{\cong} M[\frac{1}{\tilde{\xi}}]$ defines an isomorphism
$$
\varphi_M\otimes \varphi_{B^+_\crys}\colon \varphi^\ast(M\otimes_{A_\inf} B^+_\crys)\xrightarrow{\cong} M\otimes_{A_\inf}B^+_\crys.
$$

We fix a section $k\to \mathcal{O}_C/p$ of the projection $\mathcal{O}_C/p\to k$. This yields a section $W(k)\to A_\inf$ of the projection $A_\inf\to W(k)$.

\begin{definition}
\label{definition: rigidified_bkf_modules}  
A rigidified Breuil-Kisin-Fargues module is a Breuil-Kisin-Fargues module $(M,\varphi_M)$ together with a $\varphi$-equivariant isomorphism
$$
\alpha\colon M\otimes_{A_\inf}B^+_\crys\xrightarrow{\simeq} (M\otimes_{A_\inf}W(k))\otimes_{W(k)}B^+_\crys
$$
such that $\alpha$ reduces to the identity over $W(k)[\frac{1}{p}]$, i.e., 
$$\alpha\otimes_{B^+_\crys}W(k)[\frac{1}{p}]=\Id_{M\otimes_{A_\inf}W(k)[\frac{1}{p}]}.
$$
\end{definition}
We call an isomorphism $\alpha$ as in \Cref{definition: rigidified_bkf_modules} a rigidification of $(M,\varphi_M)$.
We denote by
$$
\mathrm{BKF}_{\mathrm{rig}}
$$
the category of rigidified Breuil-Kisin-Fargues modules, i.e., its objects are rigidified Breuil-Kisin-Fargues modules and the morphisms are morphisms of Breuil-Kisin-Fargues modules respecting the given rigidifications.
We illustrate the effect of imposing a rigidification in the case of Breuil-Kisin-Fargues modules of rank $1$.

\begin{lemma}
\label{lemma:bkf-modules-of-rank-1}
Each finite free Breuil-Kisin-Fargues modules of rank $1$ is isomorphic to
 the Breuil-Kisin-Fargues module $A_\inf\{d\}$ for some $d\in \Z$. For $d,d^\prime\in \Z$ we have
$$
\mathrm{Hom}_{\mathrm{BKF}}(A_\inf\{d\},A_\inf\{d^\prime\})\cong
\begin{cases}
\Z_p & \text{ if } d\leq d^\prime\\
0 & \text{ otherwise }  
\end{cases}
$$
while
$$
\mathrm{Hom}_{\mathrm{BKF}_{\mathrm{rig}}}(A_\inf\{d\},A_\inf\{d^\prime\})\cong
\begin{cases}
\Z_p & \text{ if } d= d^\prime\\
0 & \text{ otherwise }  
\end{cases}
$$
\end{lemma}
\begin{proof}
The first statement and the computation of $\mathrm{Hom}_{\mathrm{BKF}}(A_{\inf}\{d\},A_\inf\{d^\prime\})$ follows from \Cref{theorem: equivalent_descriptions_of_bkf_modules} (and the example following it). To prove
$$
\mathrm{Hom}_{\mathrm{BKF}_{\mathrm{rig}}}(A_\inf\{d\},A_\inf\{d^\prime\})=0
$$ 
if $d\neq d^\prime$ it suffices to show that the isocrystal $(A_\inf\{d\}\otimes_{A_\inf}W(k)[\frac{1}{p}],\varphi_{A_{\inf}\{d\}}\otimes \mathrm{Id})$ has slope $d$ because there are no morphisms between isocrystals of different slopes.
But the image of $$\tilde{\xi}=\frac{[\epsilon^p]-1}{[\epsilon]-1}=1+[\epsilon]+\ldots + [\epsilon^{p-1}] $$ in $W(k)$ is $p$ because $k$ does not contain non-trivial $p$-power roots of unity.
Hence, the isocrystal $(A_\inf\{d\}\otimes_{A_\inf}W(k)[\frac{1}{p}],\varphi_{A_{\inf}\{d\}}\otimes \mathrm{Id})$ is isomorphic to $(W(k)[\frac{1}{p}],p^{d}\varphi)$ which is of slope $d$. 
\end{proof}

The category $\mathrm{BKF}_{\mathrm{rig}}$ of rigidified Breuil-Kisin-Fargues modules is again an exact tensor category with the exact structure and tensor product inherited from $\mathrm{BKF}$ (taking the tensor product of the rigidifications), but it is moreover abelian (cf.\ \Cref{theorem: rigidified_bkf_modules_abelian}) contrary to the case of the category $\mathrm{BKF}$ of non-rigidified Breuil-Kisin-Fargues modules.

To prove this we recall the following criterion for an $A_\inf$-module $M$ to satisfy that $M[\frac{1}{p}]$ is finite projective over $A_\inf[\frac{1}{p}]$.
Recall the element $\mu=[\epsilon]-1\in A_\inf$ (cf.\ the notation in the beginning of \Cref{section: breuil_kisin_fargues_modules_up_to_isogeny}).

\begin{lemma}
\label{lemma: criterion_for_a_inf_modules_to_be_free_after_inverting_p}
Let $M$ be a finitely presented $A_\inf$-module. Assume
\begin{itemize}
\item[i)] $M[\frac{1}{p\mu}]$ is finite projective over $A_\inf[\frac{1}{p\mu}]$.
\item[ii)] $M\otimes_{A_\inf}B^+_{\crys}$ is finite projective over $B^+_\crys$. 
\end{itemize}
Then $M[\frac{1}{p}]$ is finite free over $A_{\inf}[\frac{1}{p}]$.
\end{lemma}
\begin{proof}
See \cite[Lemma 4.19.]{bhatt_morrow_scholze_integral_p_adic_hodge_theory}.  
\end{proof}

We record the following corollary in the case of Breuil-Kisin-Fargues modules.

\begin{corollary}
\label{corollary: cokernel_of_bkf_modules_finite_projective_after_inverting_mu} 
Let $f\colon (M,\varphi_M)\to (N,\varphi_N)$ be a morphism of Breuil-Kisin-Fargues modules and let $Q:=\mathrm{Coker}(M\xrightarrow{f} N)$ be the cokernel (of $A_\inf$-modules). Then
$Q[\frac{1}{p\mu}]$ is a finite free $A_\inf[\frac{1}{p\mu}]$-module.
\end{corollary}
\begin{proof}
By \cite[Lemma 4.26.]{bhatt_morrow_scholze_integral_p_adic_hodge_theory} there is a canonical isomorphism
$$
M\otimes_{A_\inf}A_\inf[\frac{1}{\mu}]\cong T_M\otimes_{\Z_p}A_\inf[\frac{1}{\mu}]
$$
with $T_M:=(M\otimes_{A_\inf}W(C^\flat))^{\varphi_M=1}$
and similarly for $N$,
$$
N\otimes_{A_\inf}A_\inf[\frac{1}{\mu}]\cong T_N\otimes_{\Z_p}A_\inf[\frac{1}{\mu}],
$$
with $T_N:=(N\otimes_{A_\inf}W(C^\flat))^{\varphi_N=1}$.
In particular, we get
$$
Q\otimes_{A_\inf}A_{\inf}[\frac{1}{p\mu}]\cong Q^\prime\otimes_{\Q_p}A_{\inf}[\frac{1}{p\mu}]
$$
where $Q^\prime$ is the cokernel of the morphism $T_M\otimes_{\Z_p}\Q_p\to T_N\otimes_{\Z_p}\Q_p$ induced by $f$.
Thus we see that
$$
Q\otimes_{A_\inf}A_{\inf}[\frac{1}{p\mu}]
$$
is finite free.
\end{proof}

We remark that in general $Q$ (with the induced Frobenius from $(N,\varphi_N)$) will not be a Breuil-Kisin-Fargues modules, because $Q[\frac{1}{p}]$ need not be a finite projective (equivalently, free) $A_\inf[\frac{1}{p}]$-module. The purpose of introducing the rigidifications is exactly to ensure this freeness (cf.\ the remark after \cite[Lemma 4.27.]{bhatt_morrow_scholze_integral_p_adic_hodge_theory}).  

\begin{theorem}
\label{theorem: rigidified_bkf_modules_abelian}  
The category $\mathrm{BKF}_{\mathrm{rig}}$ of rigidified Breuil-Kisin-Fargues modules is an abelian tensor category.
\end{theorem}
\begin{proof}
It suffices to prove that $\mathrm{BKF}_{\mathrm{rig}}$ is abelian.
Let $(M,\varphi_M,\alpha_M)$ and $(N,\varphi_N,\alpha_N)$ be two rigidified Breuil-Kisin-Fargues modules with Frobenii $\varphi_M$ resp. $\varphi_N$ (after inverting $\tilde{\xi}$) and rigidifications $\alpha_M$ and $\alpha_N$. Let $f\colon (M,\varphi_M,\alpha_M)\to (N,\varphi_N,\alpha_N)$ be a morphism of rigidified Breuil-Kisin-Fargues modules. Consider the exact sequence
$$
0\to K\to M\xrightarrow{f}N\to Q\to 0
$$
of $A_\inf$-modules. We will show that the kernel $K$ and the cokernel $Q$ are finitely presented over $A_\inf$ and admit Frobenii $\varphi_K$, $\varphi_Q$ (after inverting $\tilde{\xi}$) and rigidifications $\alpha_K$, $\alpha_{Q}$ induced from $\varphi_M, \varphi_N$ resp.\ $\alpha_M,\alpha_N$, such that $(K,\varphi_K,\alpha_K)$ and $(Q,\varphi_Q,\alpha_Q)$ are rigidified Breuil-Kisin-Fargues modules. This will then imply that the category $\mathrm{BKF}_{\mathrm{rig}}$ is abelian. In particular, we have to prove that $K[\frac{1}{p}]$ and $Q[\frac{1}{p}]$ are finite free.
As $f$ commutes with $\varphi_M$ resp.\ $\varphi_N$ it is clear that $K$ and $Q$ are canonically equipped with Frobenii $\varphi_K$ and $\varphi_Q$.
Clearly, the module $Q$ is finitely presented. By \Cref{corollary: cokernel_of_bkf_modules_finite_projective_after_inverting_mu} the module $Q[\frac{1}{p\mu}]$ is finite projective over $A_\inf[\frac{1}{p\mu}]$. Moreover, as $f$ respects the rigidifications
$$
\alpha_M\colon M\otimes_{A_\inf}B^+_{\crys}\cong \overline{M}\otimes_{W(k)}B^+_\crys
$$
resp.\
$$
\alpha_N\colon N\otimes_{A_\inf}B^+_{\crys}\cong \overline{N}\otimes_{W(k)}B^+_\crys,
$$
where $\overline{M}:=M\otimes_{A_\inf}W(k)$ resp.\ $\overline{N}:=N\otimes_{A_\inf}W(k)$,
 we see that $Q\otimes_{A_\inf}B^+_{\crys}$ is isomorphic to the cokernel of
$$
\overline{M}\otimes_{W(k)}B^+_\crys\to \overline{N}\otimes_{W(k)}B^+_\crys,
$$
which is free because the cokernel of $\overline{M}[\frac{1}{p}]\to \overline{N}[\frac{1}{p}]$ is free as $W(k)[\frac{1}{p}]$ is a field. By \Cref{lemma: criterion_for_a_inf_modules_to_be_free_after_inverting_p} we can conclude that $Q[\frac{1}{p}]$ is finite free over $A_\inf[\frac{1}{p}]$. As moreover, $M[\frac{1}{p}]$ and $N[\frac{1}{p}]$ are finite free over $A_\inf[\frac{1}{p}]$, we can conclude that $K[\frac{1}{p}]$ is finite projective (and hence finite free).
Let $C$ be the two-term complex (concentrated in degree 0 and 1) 
$$
\ldots\to M\xrightarrow{f} N\to \ldots
$$
which is a perfect complex of $A_\inf$-modules satisfying that its cohomology if finite free over $A_{\inf}[\frac{1}{p}]$. By \cite[Corollary 4.17.]{bhatt_morrow_scholze_integral_p_adic_hodge_theory} we can conclude that $H^0(C)=K$ is finitely presented.
In other words, we have proven that $(K,\varphi_K)$ and $(Q,\varphi_Q)$ are Breuil-Kisin-Fargues modules and are left with showing that they admit canonical rigidifications. By right exactness of tensor products this is clear for $Q$. But as $K,M,N,Q$ are free after inverting $p$ we get a commutative diagram
$$
\xymatrix{
0\ar[r]& K\otimes_{A_\inf}B^+_{\crys} \ar[r]\ar[d]^{\alpha_K}& M\otimes_{A_\inf}B^+_{\crys}\ar[r]\ar[d]^{\alpha_M} & N\otimes_{A_\inf}B^+_{\crys}\ar[d]^{\alpha_N}\\
0\ar[r] & \overline{K}\otimes_{W(k)}B^+_\crys\ar[r]&\overline{M}\otimes_{W(k)}B^+_\crys\ar[r]&\overline{N}\otimes_{W(k)}B^+_\crys
}
$$
whose rows are exact. In particular, we also get a canonical rigidifications $\alpha_K$ on $K$. This finishes the proof of the theorem.
\end{proof}

We are now ready to define a main player of this paper.

\begin{definition}
\label{definition: rigidified_bkf_modules_up_to_isogeny}
We define the isogeny category of rigidified Breuil-Kisin-Fargues modules to be the category
$$
\mathrm{BKF}_{\mathrm{rig}}^\circ:=\Q_p\otimes_{\Z_p}\mathrm{BKF}_{\mathrm{rig}}.
$$  
\end{definition}

In other words, the category $\mathrm{BKF}_{\mathrm{rig}}^\circ$ is the Serre quotient of the abelian category $\mathrm{BKF}_{\mathrm{rig}}$ by the full subcategory of rigidified Breuil-Kisin-Fargues modules $(M,\varphi_M,\alpha_M)$ such the underlying $A_\inf$-module $M$ is annihilated by $p^n$, $n\gg 0$.

\begin{proposition}
\label{proposition: rigidified_bkf_modules_up_to_isogeny_versus_varphi_modules_after_inverting_p}
The category $\mathrm{BKF}^\circ_{\mathrm{rig}}$ of rigidified Breuil-Kisin-Fargues modules up to isogeny is equivalent to the category of triples $(N,\varphi_N,\alpha_N)$ where
\begin{itemize}
\item $N$ is a finite free $A_\inf[\frac{1}{p}]$-module
\item $\varphi_N\colon \varphi_{A_\inf[\frac{1}{p}]}^\ast(N)[\frac{1}{\tilde{\xi}}]\xrightarrow{\simeq} N[\frac{1}{\tilde{\xi}}]$ is an isomorphism and
\item $\alpha_N\colon N\otimes_{A_\inf[\frac{1}{p}]}B^+_{\crys}\cong \overline{N}\otimes_{W(k)[\frac{1}{p}]}B^+_{\crys}$ is a rigidification reducing to the identity over $W(k)[\frac{1}{p}]$ with $\overline{N}:=N\otimes_{A_\inf[\frac{1}{p}]}W(k)[\frac{1}{p}]$
\end{itemize}
such that there exists a finitely presented $A_\inf$-submodule $N^\prime\subseteq N$ satisfying $N^\prime[\frac{1}{p}]=N$ which is stable under $\varphi_N$.
\end{proposition}
\begin{proof}
The natural functor sending an object $(M,\varphi_M,\alpha_M)\in \mathrm{BKF}_{\mathrm{rig}}^\circ$ to its base extension $(M\otimes_{A_\inf}A_{\inf}[\frac{1}{p}],\varphi_M\otimes \varphi_{A_\inf[\frac{1}{p}]},\alpha_M)$ to $A_\inf[\frac{1}{p}]$ is essential surjective by the last condition on the existence of $N^\prime$. It is moreover fully faithful as inverting $p$ commutes with $\mathrm{Hom}$ for finitely presented modules.    
\end{proof}

The condition on the existence of a $\varphi$-stable $A_\inf$-lattice $N^\prime\subseteq N$ ensures that for a triple $(N,\varphi_N,\alpha_N)$ as in \Cref{proposition: rigidified_bkf_modules_up_to_isogeny_versus_varphi_modules_after_inverting_p} the ``etale realization'' (cf.\ \Cref{definition: realizations_of_bkf_modules})
$$
(N\otimes_{A_{\inf}}W(C^\flat))^{\varphi_N=1}
$$  
is a $\Q_p$-vector space of dimension the rank of $N$. Otherwise, the isocrystal $N\otimes_{A_{\inf}}W(C^\flat)$ (over $C^\flat$) need not be isoclinic of slope $0$.

We can now record the main theorem of this section.

\begin{theorem}
\label{theorem: rigidified_bkf_modules_form_tannakian_category}
The category $\mathrm{BKF}_{\mathrm{rig}}^\circ$ of rigidified Breuil-Kisin-Fargues modules up to isogeny is Tannakian.  
\end{theorem}
\begin{proof}
Theorem \Cref{theorem: rigidified_bkf_modules_abelian} implies that also the isogeny category $\mathrm{BKF}_{\mathrm{rig}}^\circ$ of rigidified Breuil-Kisin-Fargues modules is abelian, namely it identifies with the Serre quotient of the category $\mathrm{BKF}_{\mathrm{rig}}$ by the full subcategory of rigidified Breuil-Kisin-Fargues modules which are p-power torsion.
By (the proof of) \Cref{lemma: bkf_modules_rigid_exact_tensor_category} each object $(M,\varphi_M)\in\mathrm{BKF}^\circ$ is isomorphic (in the isogeny category) to some $(N,\varphi_N)$ with $N$ a finite free $A_\inf$-module. If
$$
\alpha_M\colon M\otimes_{A_\inf}B^+_\crys\cong (M\otimes_{A_\inf}W(k))\otimes_{W(k)}B^+_\crys
$$
is a rigidification for $(M,\varphi_M)$, then $\alpha_M$ also defines a rigidification of $(N,\varphi_N)$ because $M\otimes_{A_\inf}B^+_\crys\cong N\otimes_{A_\inf}B^+_\crys$ resp.\ $(M\otimes_{A_\inf}W(k))\otimes_{W(k)}B^+_\crys\cong (N\otimes_{A_\inf}W(k))\otimes_{W(k)}B^+_\crys$.
In particular, we see that every object in $\mathrm{BKF}_{\mathrm{rig}}^\circ$ admits a dual and therefore the category of $\mathrm{BKF}_{\mathrm{rig}}^\circ$ is rigid. To prove that $\mathrm{BKF}_{\mathrm{rig}}^\circ$ is Tannakian we can either apply Deligne's criterion that $\Lambda^n M$ vanishes for every $M\in \mathrm{BKF}^\circ_{\mathrm{rig}}$ and $n\gg 0$ (cf.\ \cite[Th\'eor\`eme 7.1.]{deligne_categories_tannakiennes} or use that there are various concrete fiber functors \Cref{lemma: fiber_functors_for_rigidified_bkf_modules}. 
\end{proof}

We continue by giving an alternative description of rigidified Breuil-Kisin-Fargues modules in terms of modifications of vector bundles on the Fargues-Fontaine curve.

We recall that the Fargues-Fontaine curve admits an Harder-Narasimhan formalism, namely, for every vector bundle $\mathcal{F}$ on $X_{\mathrm{FF}}$ there exists a canonical (decreasing) filtration indexed by $\lambda\in \Q$
$$
\mathrm{HN}^\lambda(\mathcal{F})\subseteq \mathcal{F}
$$
such that the associated graded pieces
$$
\mathrm{gr}^\lambda(\mathrm{HN}(\mathcal{F}))
$$
are semistable of slope $\lambda$.

\begin{theorem}
\label{theorem: rigidified_bkf_modules_versus_modifications_of_vector_bundles}  
The category $\mathrm{BKF}^\circ_{\mathrm{rig}}$ of rigidified Breuil-Kisin-Fargues modules is equivalent to the category of quadruples $(\mathcal{F},\mathcal{F}^\prime,\beta,\alpha)$ where $\mathcal{F}$, $\mathcal{F}^\prime$ are vector bundles on the Fargues-Fontaine curve $X_{\mathrm{FF}}$, $\mathcal{F}$ is trivial, $\beta\colon \mathcal{F}_{|X_{\mathrm{FF}}\setminus\{\infty\}}\cong \mathcal{F}^\prime_{|X_{\mathrm{FF}}\setminus\{\infty\}}$ is an isomorphism, and $\alpha\colon \mathcal{F}^\prime\cong \bigoplus\limits_{\lambda\in \Q}\mathrm{gr}^\lambda(\mathrm{HN}(\mathcal{F}^\prime))$ is an isomorphism inducing the identity on all graded pieces of the Harder-Narasimhan filtration. 
\end{theorem}
\begin{proof}
By \Cref{theorem: equivalent_descriptions_of_bkf_modules} (and \Cref{lemma: bkf_modules_rigid_exact_tensor_category}) the category $\mathrm{BKF}^\circ$ of Breuil-Kisin-Fargues modules up to isogeny is equivalent to the category of triples 
$$
(\mathcal{F},\mathcal{F}^\prime,\beta)
$$   
satisfying the same condition as in the statement of this theorem. Thus we are left with showing that the rigidifications can be identified under the equivalence in \Cref{theorem: equivalent_descriptions_of_bkf_modules}. Let $(M,\varphi)$ be a finite free Breuil-Kisin-Fargues module with corresponding quadruple $(\mathcal{F},\mathcal{F}^\prime,\beta,T)$ as in \Cref{theorem: equivalent_descriptions_of_bkf_modules}. Let $N$ be the $\varphi$-module over $B^+_\crys$ associated with $\mathcal{F}^\prime$ in \Cref{theorem: vector_bundles_and_varphi_modules_over_b_+_crys}. As was shown in the proof of \Cref{theorem: equivalent_descriptions_of_bkf_modules} there exists a canonical isomorphism
$$
M\otimes_{A_\inf}B^+_\crys\cong N.
$$
Recall that we have fixed a splitting $k\to \mathcal{O}_C/p$ of the projection $\mathcal{O}_C/p\to k$.
Let
$$
\mathcal{E}(-)\colon \varphi\mathrm{-Mod}_{W(k)[\frac{1}{p}]}\to \mathrm{Bun}_{X_{\mathrm{FF}}}
$$
be the exact tensor functor associated with this splitting (cf. \cite[Proposition 8.2.6.]{fargues_fontaine_courbes_et_fibres_vectoriels_en_theorie_de_hodge_p_adique}). Then for the vector bundle $\mathcal{F}^\prime$ there is a canonical isomorphism
$$
\bigoplus\limits_{\lambda\in \Q}\mathrm{gr}^\lambda(\mathrm{HN}(\mathcal{F}^\prime)\cong \mathcal{E}(N\otimes_{B^+_\crys}W(k)[\frac{1}{p}]
$$
and we see (writing $\overline{M}:=M\otimes_{A_\inf}W(k)[\frac{1}{p}]\cong N\otimes_{B^+_\crys}W(k)[\frac{1}{p}]$) that the set of isomorphisms
$$
\alpha_M\colon M\otimes_{A_\inf}B^+_\crys\cong \overline{M}\otimes_{W(k)[\frac{1}{p}]}B^+_\crys 
$$
is in bijection with the set of isomorphisms
$$
\alpha\colon \mathcal{F}^\prime\cong \bigoplus\limits_{\lambda\in \Q}\mathrm{gr}^\lambda(\mathrm{HN}(\mathcal{F}^\prime).
$$
Clearly, the condition on $\alpha_M$ resp.\ $\alpha$ to reduce to the identity over $W(k)[\frac{1}{p}]$ resp.\ the graded pieces of the Harder-Narasimhan filtration is preserved.
This finishes the proof.
\end{proof}

We note that we actually proved that for a finite free Breuil-Kisin-Fargues module $(M,\varphi_M)$ with corresponding modification 
$$(\mathcal{F},\mathcal{F}^\prime,\alpha\colon \mathcal{F}_{|X_{\mathrm{FF}}\setminus\{\infty\}}\cong \mathcal{F}^\prime_{|X_{\mathrm{FF}}\setminus\{\infty\}},T\subseteq H^0(X_{\mathrm{FF}},\mathcal{F}))
$$ 
as in \Cref{theorem: equivalent_descriptions_of_bkf_modules} there is a bijection
$$
\begin{matrix}
\{\text{rigidifications }M\otimes_{A_\inf}B^+_{\crys}\cong (M\otimes_{A_\inf}W(k))\otimes_{W(k)}B^+_{\crys}\} \\\cong \{\textrm{rigidifications }\mathcal{F}^\prime\cong \bigoplus\limits_{\lambda\in \Q}\mathrm{gr}^\lambda(\mathrm{HN}(\mathcal{F}^\prime))\}
\end{matrix}
$$
where we call an isomorphism $\mathcal{F}^\prime\cong \bigoplus\limits_{\lambda\in \Q}\mathrm{gr}^\lambda(\mathrm{HN}(\mathcal{F}^\prime))$ reducing to the identity on the graded pieces of the Harder-Narasimhan filtration a rigidification of the data $(\mathcal{F},\mathcal{F}^\prime, \alpha,T)$ or just of $\mathcal{F}^\prime$.
In particular, if $\mathcal{F}^\prime$ happens to be semistable there exists a unique rigidification on $(M,\varphi_M)$, namely the one given by the identity on $\mathcal{F}^\prime$.

Using \Cref{theorem: rigidified_bkf_modules_versus_modifications_of_vector_bundles} it is possible to give an alternative proof of \Cref{theorem: rigidified_bkf_modules_form_tannakian_category}. Namely, the category of semistable vector bundles of a fixed slope is abelian and therefore the cokernel of a morphism between vector bundles on the Fargues-Fontaine curve which respects fixed splittings of the respective Harder-Narasimhan filtrations is again a vector bundle. This proves (using \Cref{theorem: rigidified_bkf_modules_versus_modifications_of_vector_bundles}) that the category $\mathrm{BKF}^\circ_{\mathrm{rig}}$ is abelian. Moreover, as it is a rigid tensor category, the existence of a fiber functor (cf.\ \Cref{lemma: fiber_functors_for_rigidified_bkf_modules}) or Deligne's criterion proves that it is moreover Tannakian.

\Cref{theorem: rigidified_bkf_modules_versus_modifications_of_vector_bundles} also proves that the category $\mathrm{BKF}^\circ_{\mathrm{rig}}$ of rigidified Breuil-Kisin-Fargues modules up to isogeny does not depend (up to canonical isomorphism) on the choice of the section $k\to \mathcal{O}_C/p$.

% \begin{lemma}
% \label{lemma:extension-of-rigidification}
% Let $0\to (M,\varphi_M)\xrightarrow{f} (E,\varphi_E)\xrightarrow{g} (N,\varphi_N)\to 0$ be an extension of finite free Breuil-Kisin-Fargues modules. Let 
% $\alpha_M$ and $\alpha_N$ be rigidifications for $(M,\varphi_M)$ resp.\ $(N,\varphi_N)$. Then there exists a unique rigidification on $(E,\varphi_E)$ such that $f,g$ are morphisms of rigidified Breuil-Kisin-Fargues modules.
% \end{lemma}
% \begin{proof}
% By \Cref{theorem: rigidified_bkf_modules_versus_modifications_of_vector_bundles} and the remark following it it suffices to prove that for a short exact sequence
% $$
% 0\to \mathcal{F}^\prime\to \mathcal{F}\to \mathcal{F}^{\prime\prime}\to 0
% $$
% to vector bundles on the Fargues-Fontaine curve and rigidifications for $\mathcal{F}^\prime$ and $\mathcal{F}^{\prime\prime}$ there exists a unique rigidification on $\mathcal{F}$ (note that we are using that the passage from Breuil-Kisin-Fargues modules to modifications on the Fargues-Fontaine curve is an exact operation, which is true in this case.)

% a   
% \end{proof}

We do not know how to describe the category of rigidified Breuil-Kisin-Fargues modules up to isogeny only in terms of pairs $(V,\Xi)$ where $V$ is a finite-dimensional $\Q_p$-vector space and $\Xi\subseteq V\otimes_{\Q_p}B_{\dR}$ a $B_{\dR}^+$-lattice. However, we can record the following.

\begin{lemma}
\label{lemma:cokernel-of-rigidified-morphism-between-pairs}
Let $f\colon (M,\varphi_M,\alpha_M)\to (M^\prime,\varphi_{M^\prime},\alpha_{M^\prime})\in \mathrm{BKF}_{\mathrm{rig}}$ be a morphism of finite free rigidified Breuil-Kisin-Fargues modules. Let $g:(T,\Xi)\to (T^\prime,\Xi^\prime)$ be the associated morphism between pairs of a $\Z_p$- and $B^+_\dR$-lattice as in \Cref{theorem: equivalent_descriptions_of_bkf_modules}. Then the cokernel of the morphism
$$
\Xi\xrightarrow{g} \Xi^\prime
$$  
is free over $B^+_\dR$.
\end{lemma}
\begin{proof}
By construction of $(T,\Xi)$ resp.\ $(T^\prime,\Xi^\prime)$ the morphism $g\colon \Xi\to \Xi^\prime$ is given by
$$
f\otimes \mathrm{Id}\colon M\otimes_{A_\inf}B^+_{\dR}\to N\otimes_{A_\inf}B^+_\dR.
$$
Set $\overline{M}:=M\otimes_{A_\inf}W(k)$ resp.\ $\overline{N}:=N\otimes_{A_\inf}W(k)$. By assumption the diagram
$$
\xymatrix{
M\otimes_{A_\inf}B^+_{\dR}\ar[r]^{f\otimes\mathrm{Id}}\ar[d]^{\alpha_M\otimes B^+_\dR} & N\otimes_{A_\inf}B^+_\dR\ar[d]^{\alpha_N\otimes B^+_\dR}\\
\overline{M}\otimes_{A_\inf}B^+_{\dR}\ar[r]^{\overline{f}\otimes\mathrm{Id}} & \overline{N}\otimes_{A_\inf}B^+_\dR.
}
$$
with both vertical morphisms isomorphisms commutes. In particular, the cokernel of $g$ is free.
\end{proof}

Let us list some fiber functors, i.e., ``realizations'', of the Tannakian category $\mathrm{BKF}_{\mathrm{rig}}^\circ$.

We recall the following lemma which is a direct corollary of the classification of vector bundles on the Fargues-Fontaine curve (cf.\ \cite[Th\'eor\`eme 8.2.10.]{fargues_fontaine_courbes_et_fibres_vectoriels_en_theorie_de_hodge_p_adique}).

\begin{lemma}
\label{lemma: graded_vector_bundles_and_isocrystals}
The category of $\varphi$-modules over $W(k)[\frac{1}{p}]$, i.e., the category of isocrystals over $k$, is equivalent to the category of $\Q$-graded vector bundles $\mathcal{E}=\bigoplus\limits_{\lambda\in \Q}\mathcal{E}^\lambda$ such that $\mathcal{E}^\lambda$ is semistable of slope $\lambda$.  
\end{lemma}
\begin{proof}
Cf.\ \cite[Lemma 3.6.]{anschuetz_reductive_group_schemes_over_the_fargues_fontaine_curve}.  
\end{proof}

\begin{definition}
\label{definition: realizations_of_bkf_modules}
Let $(M,\varphi_M)$ be a Breuil-Kisin-Fargues module. We define
\begin{itemize}
\item its ``\'etale realization'' $T:=(M\otimes_{A_\inf}W(C^\flat))^{\varphi_M=1}$, a $\Z_p$-module.
\item its ``crystalline realization'' $D:=M\otimes_{A_\inf}W(k)$, a $W(k)$-module equipped with the $\varphi$-semilinear isomorphism $\varphi_D:=\varphi_M\otimes W(k)[\frac{1}{p}]:\varphi^\ast(D[\frac{1}{p}])\cong D[\frac{1}{p}]$ after inverting $p$.
\item and its ``de Rham realization'' $V:=M\otimes_{A_\inf,\theta}\mathcal{O}_C$, a $\mathcal{O}_C$-module. 
\end{itemize}
\end{definition}

\begin{lemma}
\label{lemma: fiber_functors_for_rigidified_bkf_modules}
The realizations from \Cref{definition: realizations_of_bkf_modules} define fiber functors
\begin{itemize}
\item $\omega_{\acute{e}t}\colon\mathrm{BKF}^\circ_{\mathrm{rig}}\to \mathrm{Vec}_{\Q_p}$
\item $\omega_{\crys}\colon \mathrm{BKF}^\circ_{\mathrm{rig}}\to \varphi\mathrm{-Mod}_{W(k)[\frac{1}{p}]}$
\item $\omega_{\dR}\colon \mathrm{BKF}^\circ_{\mathrm{rig}}\to \mathrm{Vec}_C$
\end{itemize}
 on the Tannakian category $\mathrm{BKF}^\circ_{\mathrm{rig}}$ of rigidified Breuil-Kisin-Fargues modules up to isogeny.
\end{lemma}
\begin{proof}
The functors are defined by sending a Breuil-Kisin-Fargues module to the respective realizations (with $p$ inverted).
As each Breuil-Kisin-Fargues modules is isogenous to a finite free Breuil-Kisin-Fargues module all three functors are exact. The functors given by crystalline and de Rham realization are clearly tensor functor. To prove that also the \'etale realization $\omega_{\acute{e}t}$ defines a tensor functor we remark that it factors canonically into the tensor functor
$$
\mathrm{BKF}^\circ_{\mathrm{rig}}\to \varphi\mathrm{-Mod}_{W(C^\flat)[\frac{1}{p}]},\ (M,\varphi_M)\mapsto (M\otimes_{A_\inf}W(C^\flat)[\frac{1}{p}],\varphi_M\otimes \mathrm{Id}),
$$ 
to isocrystals over $C^\flat$ and the tensor functor
$$
\varphi\mathrm{-Mod}_{W(C^\flat)[\frac{1}{p}]}\to \mathrm{Vec}_{\Q_p}
$$
sending an isocrystal to its isoclinic part of slope $0$.
\end{proof}

Actually, the \'etale realization $\omega_{\acute{e}t}\otimes_{\Q_p}C$, base changed to $C$, underlies a \textit{filtered} fiber functor (cf.\ \cite{saavedra_rivano_categories_tannakiennes} and \cite{ziegler_graded_and_filtered_fiber_functors}) on the category $\mathrm{BKF}^\circ_{\mathrm{rig}}$.
This nicely parallels the case of rational Hodge structures.

\begin{definition}
\label{definition:filtration-on-etale-realization}
Let $(M,\varphi_M)$ be a finite free Breuil-Kisin-Fargues modules with associated \'etale realization
$$
T:=(M\otimes_{A_\inf}W(C^\flat))^{\varphi_M=1}
$$
 and $B^+_\dR$-lattice 
$$\Xi:=M\otimes_{A_\inf}B^+_\dR\subseteq T\otimes_{\Z_p}B_\dR\cong M\otimes_{A_\inf}B_\dR.
$$
Then we define the decreasing filtration $\mathrm{Fil}^j(T\otimes_{\Z_p}C), j\in \Z,$ on 
$$
T\otimes_{\Z_p}C=T\otimes_{\Z_p}B^+_\dR/(T\otimes_{\Z_p}\xi B^+_\dR)
$$ 
by
$$
\mathrm{Fil}^j(T\otimes_{\Z_p}C):=\mathrm{Im}(\xi^j\Xi\cap T\otimes_{\Z_p}B^+_\dR\to T\otimes_{\Z_p}C)
$$  
with $j\in\Z$.
\end{definition}

Written in a suitable basis $e_1,\ldots,e_n$ of $T\otimes_{\Z_p}B^+_{\dR}$ (over $B^+_\dR$) there exist $\lambda_1,\ldots,\lambda_n\in \Z$ satisfying $\lambda_1\geq\ldots\geq \lambda_n$ such that the lattice $\Xi\subseteq T\otimes_{\Z_p}B_{\dR}$ is generated by the $\xi^{\lambda_1}e_1,\ldots,\xi^{\lambda_n}e_n$.
Then the filtration $\mathrm{Fil}^j(T\otimes_{\Z_p}C)$ is given by the image of
$$
\langle\xi^{\lambda_1+j}e_1,\ldots, \xi^{\lambda_n+j}e_n\rangle \cap T\otimes_{\Z_p}B^+_\dR
$$
in $T\otimes_{\Z_p}C$, i.e., the $C$-subspace generated by the residue classes $\overline{e}_i$ of the $e_i$ such that $\lambda_i+j\leq 0$. 

\begin{lemma}
\label{lemma: filtered_fiber_functor_on_rigidified_bkf_modules}
The fiber functor
$$
\omega_{\acute{e}t}\otimes_{\Q_p}C\colon \mathrm{BKF}^\circ_{\mathrm{rig}}\to \mathrm{Vec}_C
$$
admits a canonical filtration induced by the filtration in \Cref{definition:filtration-on-etale-realization}, i.e., the functor
$$
\omega\colon \mathrm{BKF}^\circ_{\mathrm{rig}}\to \mathrm{FilVec}_C, M\mapsto (\omega_{\acute{e}t}\otimes_{\Q_p}C, \mathrm{Fil}^\bullet(\omega_{\acute{e}t}(M)\otimes_{\Q_p}C)
$$ 
is exact.
\end{lemma}
\begin{proof}
Writing down adapted bases as was done following \Cref{definition:filtration-on-etale-realization} one sees that $\omega$ is a tensor functor.
As symmetric powers and exterior powers are kernels of projectors (in characteristic $0$) and the category $\mathrm{FilVec}_C$ of filtered vector spaces is pseudo-abelian we see that $\omega$ commutes with symmetric and exterior powers. Moreover, $\omega$ commutes with duals.
Let $0\to M\to N\to Q\to 0$ be an exact sequence in $\mathrm{BKF}^\circ_{\mathrm{rig}}$. We have to proof that the sequence
$$
0\to \mathrm{gr}(\omega(M))\to \mathrm{gr}(\omega(N))\to \mathrm{gr}(\omega(Q))\to 0
$$
is exact. Considering dimensions and taking duals it suffices to prove that the left arrow $\mathrm{gr}(\omega(M))\to \mathrm{gr}(\omega(N))$ is injective. This is equivalent to the statement that
$$
\Lambda^r(\mathrm{gr}(\omega(M)))\to \Lambda^r(\mathrm{gr}(\omega(N)))
$$
is non-zero where $r:=\mathrm{rk}(M)$. As $\omega$ commutes with exterior powers we may replace the morphism $M\to N$ by $\Lambda^r M\to \Lambda^r N$ and assume that $r=1$.  By \Cref{lemma:cokernel-of-rigidified-morphism-between-pairs}
Let $(V,\Xi)$ resp.\ $(V^\prime,\Xi^\prime)$ be the associated pairs of a $\Q_p$- and a $B^+_\dR$-lattice. We identify $V,\Xi$ with their images in $V^\prime$ resp.\ $\Xi^\prime$. Tensoring with the dual of $(V,\Xi)$ we can assume that $\Xi=V\otimes_{\Q_p}B^+_\dR$. Pick a generator $v\in V$. Then the element $v\otimes 1$ is an element in $\Xi^\prime$ and moreover, it is part of a basis of $\Xi^\prime$ by \Cref{lemma:cokernel-of-rigidified-morphism-between-pairs}. Therefore we can, cf.\ the description following \Cref{definition:filtration-on-etale-realization}, find an adapted basis $v:=e_1,\ldots, e_n$ of $V^\prime\otimes B^+_\dR$ containing $v$ such that $\Xi^\prime=\langle e_1,\xi^{\lambda_2}e_2,\ldots,\xi^{\lambda_n}e_n\rangle$. From the concrete description given after \Cref{definition:filtration-on-etale-realization} it follows that the image of $v$ in $\mathrm{gr}^0(V^\prime\otimes_{\Q_p}C)$ is non-zero. This finishes the proof.
\end{proof}

We remark that sending a finite free Breuil-Kisin-Fargues modules $M\in \mathrm{BKF}^\circ$ (without fixing a rigidification) to the filtered vector space $\omega_{\acute{e}t}(M)\otimes_{\Q_p}C$ is not an exact operation. An counterexample is provided by the morphism
$$
(\Q_p,\xi B^+_\dR)\to (\Q_p,B^+_\dR).
$$

\begin{definition}
\label{definition: motivic_galois_group}
Let $\omega_?\in \{\omega_{\acute{e}t}, \omega_{\crys}, \omega_{\dR}\}$ be one of the fiber functors in \Cref{lemma: fiber_functors_for_rigidified_bkf_modules}. Then we set
$$
G_{?}:=\mathrm{Aut}^\otimes(\omega_{?})
$$
as the group of tensor automorphisms of $\omega_{?}$.    
\end{definition}

We do not know how to describe the affine group schemes $G_{\acute{e}t}, G_{\crys}$ or $G_{\dR}$ apart from their maximal torus quotients corresponding to the full subcategory of rigidified Breuil-Kisin-Fargues modules admitting complex multiplication.
We record the following general result about these group schemes.

\begin{lemma}
\label{lemma:band-of-bkf-modules-connected}
The band of the Tannakian category $\mathrm{BKF}^\circ_{\mathrm{rig}}$ is connected. In particular, the group schemes $G_{\acute{e}t}, G_{\crys}$ and $G_{\dR}$ are connected.  
\end{lemma}
\begin{proof}
It suffices to show that every object $M\in \mathrm{BKR}^\circ_{\mathrm{rig}}$ such that the smallest abelian subcategory $\langle M\rangle$, i.e., the full subcategory of subquotients of a finite direct sum $\bigoplus M$, is closed under the tensor product is a direct sum of the unit object.
Assume that there exists an object like this. Set $V:=\omega_{\acute{e}t}(M)=(M\otimes_{A_\inf}W(C^\flat)^{\varphi_M=1}[\frac{1}{p}]$. Then the filtration of \Cref{lemma: filtered_fiber_functor_on_rigidified_bkf_modules} on the \'etale realization $\omega_{\acute{e}t}(M)\otimes_{\Q_p}C=V\otimes_{\Q_p}C$ must be trivial. Set $\Xi:=M\otimes_{A_\inf}B^+_{\dR}\subseteq V\otimes_{\Q_p}B_{\dR}$. In a suitable basis $e_1,\ldots,e_n$ of $V\otimes_{\Q_p}B^+_\dR$ we can write $\Xi=\langle\xi^{\lambda_1}e_1,\ldots, \xi^{\lambda_n}e_n\rangle$ with $\lambda_1,\ldots,\lambda_n\in \Z$. Then for $j\in \Z$ 
$$
\xi^j\Xi\subseteq V\otimes_{\Q_p}B^+_\dR
$$
is generated by the $\xi^{\mathrm{max}\{\lambda_i+j,0\}}e_i$. As the filtration on $V\otimes_{\Q_p}C$ is trivial we can conclude that
$$
\lambda_i\leq 0
$$
and
$$
\lambda_i+1>0
$$
for all $i$. In other words, $\lambda_i=0$ for all $i$, i.e., $\Xi=V\otimes_{\Q_p}B^+_{\dR}$, which proves that $M$ is a direct sum of the unit object of $\mathrm{BKF}^\circ_{\mathrm{rig}}.$
\end{proof}

We now calculate the Ext-groups
$$
\mathrm{Ext}^1_{\mathrm{BKF}^\circ_{\mathrm{rig}}}(A_\inf\{d\},A_\inf\{d^\prime\})
$$
for $d,d^\prime\in \Z$. In particular, as they will turn out to be non-zero we will be able to conclude that the category $\mathrm{BKF}^\circ_{\mathrm{rig}}$ of Breuil-Kisin-Fargues modules is not semisimple (and thus the group schemes $G_{\acute{e}t},G_{\crys},G_{\dR}$ are not reductive).

We recall one more period ring, namely $B_e$. By definition,
$$
B_e:=H^0(X_{\mathrm{FF}},\mathcal{O}_{X_\mathrm{FF}})=B_\crys^{\varphi=1}.
$$
Moreover, there exists the ``fundamental sequence of $p$-adic Hodge theory'' involving $B_e$:
$$
0\to \Q_p\to B_e\to B_\dR/(B_\dR^+)\to 0
$$
(cf.\ this follows from \cite[Exemple 6.4.2.]{fargues_fontaine_courbes_et_fibres_vectoriels_en_theorie_de_hodge_p_adique}).

\begin{lemma}
\label{lemma: extensions_of_bkf_modules_of_rank_one}
Let $d\in \Z$. Then 
$$
\mathrm{Ext}^1_{\mathrm{BKF}^\circ_{\mathrm{rig}}}(A_\inf,A_\inf\{d\})\cong
B_{\dR}/{t^d B^+_{\dR}}.
$$
\end{lemma}
\begin{proof}
We use \Cref{theorem: rigidified_bkf_modules_versus_modifications_of_vector_bundles} and will classify triples $(\mathcal{F},\mathcal{F}^\prime,\beta)$ fitting into a commutative diagramm with exact rows (where a dotted arrow means an isomorphism outside $\infty\in X_{\mathrm{FF}}$)
$$
\xymatrix{
0 \ar[r] & \mathcal{O}_{X_{\mathrm{FF}}} \ar[r]\ar@{-->}[d]^{t^{d}} & \mathcal{F} \ar[r]\ar@{-->}[d]^{\beta} & \mathcal{O}_{X_{\mathrm{FF}}}\ar[r]\ar[d]^{=} & 0 \\
0 \ar[r] & \mathcal{O}_{X_{\mathrm{FF}}}(d) \ar[r] & \mathcal{F}^\prime \ar[r] & \mathcal{O}_{X_{\mathrm{FF}}}\ar[r] & 0 
}
$$
and isomorphisms
$$
\alpha\colon \mathcal{F}^\prime\cong \mathrm{gr}^\bullet(\mathrm{HN}(\mathcal{F}^\prime))
$$
inducing the canonical rigidifications on $\mathcal{O}_{X_{\mathrm{FF}}}(d)$ resp.\ $\mathcal{O}_{X_{\mathrm{FF}}}$.
First we note that this last requirement actually implies that 
$$
\mathcal{F}^\prime\cong \mathcal{O}_{X_{\mathrm{FF}}}(d)\oplus \mathcal{O}_{X_{\mathrm{FF}}}.
$$
Moreover, 
$$
\mathcal{F}\cong \mathcal{O}_{X_{\mathrm{FF}}}\oplus\mathcal{O}_{X_{\mathrm{FF}}}.
$$
In other words, an extension of $\mathrm{A}_\inf$ by $\mathrm{A}_{\inf}\{d\}$ is thus determined by an automorphism, preserving the factor $\mathcal{O}_{X_{\mathrm{FF}}}(d)$,
$$
\mathcal{O}_{X_{\mathrm{FF}}}(d)\oplus \mathcal{O}_{X_{\mathrm{FF}}}
$$
restricting to the identity on both factors, i.e., by an element in
$$
\mathrm{Hom}_{X_{\mathrm{FF}}}(\mathcal{O}_{X_{\mathrm{FF}}},\mathcal{O}_{X_{\mathrm{FF}}}(d))\cong (B^+_\crys)^{\varphi=p^{d}},
$$
and an isomorphism
$$
\beta\colon (\mathcal{O}_{X_{\mathrm{FF}}}\oplus\mathcal{O}_{X_{\mathrm{FF}}})_{|X_{\mathrm{FF}}\setminus\{\infty\}}\cong (\mathcal{O}_{X_{\mathrm{FF}}}(d)\oplus \mathcal{O}_{X_{\mathrm{FF}}})_{|X_{\mathrm{FF}}\setminus\{\infty\}}
$$
on $X_{\mathrm{FF}}\setminus\{\infty\}$ (which again must preserve the filtration and associated gradeds). 
We thus obtain a canonical surjection (of $\Q_p$-vector spaces)
$$
\Gamma\colon B_et^{d}\oplus (B_{\crys}^+)^{\varphi=p^{d}}\to \mathrm{Ext}^1_{\mathrm{BKF}^\circ_{\mathrm{rig}}}(A_\inf,A_\inf\{d\})
$$
by sending $(a,b)\in B_et^{d}\oplus (B_{\crys}^+)^{\varphi=p^{d}}$ to the quadruple
$$
(\mathcal{O}_{X_{\mathrm{FF}}}\oplus \mathcal{O}_{X_{\mathrm{FF}}}, \mathcal{O}_{X_{\mathrm{FF}}}(d)\oplus \mathcal{O}_{X_{\mathrm{FF}}},\beta=
\begin{pmatrix}
  t^{d} & a \\
0 & 1
\end{pmatrix},\alpha=
\begin{pmatrix}
1 & b\\
0 & 1  
\end{pmatrix})
$$
Assume that a pair $(a,b)$ defines a trivial extension. Then there exists 
$$
c\in \Q_p\cong H^0(X_{\mathrm{FF}},\mathcal{O}_{X_{\mathrm{FF}}})
$$ 
and 
$$
c^\prime\in (B^+_\crys)^{\varphi=p^d}\cong H^0(X_{\mathrm{FF}},\mathcal{O}_{X_{\mathrm{FF}}}(d))
$$ 
such that
$$
\begin{matrix}
  \begin{pmatrix}
    t^d & a\\
    0 & 1
  \end{pmatrix}
  \begin{pmatrix}
    1 & c \\
    0 & 1 
  \end{pmatrix}
=
\begin{pmatrix}
  1 & c^\prime \\
  0 & 1
\end{pmatrix}
\begin{pmatrix}
  t^d & 0 \\
  0 & 1
\end{pmatrix} \\
\begin{pmatrix}
  1 & b \\
  0 & 1
\end{pmatrix}
\begin{pmatrix}
  1 & c^\prime  \\
  0 & 1
\end{pmatrix}
=
\begin{pmatrix}
  1 & 0 \\
  0 & 1
\end{pmatrix}.
\end{matrix}
$$
In other words, 
$$
b=-c^\prime
$$ 
and
$$
ct^d+a=c^\prime=-b.
$$
For $b\in (B^+_\crys)^{\varphi=p^d}$ the element $\frac{b}{t^d}$ lies in $B_e$. Therefore the pair 
$$
(\frac{b}{t^d}t^d,b)\in \mathrm{Ker}(\Gamma)
$$
lies in the kernel of $\Gamma$.
In particular, we see that the morphism
$$
\Gamma_{|B_et^d}\colon B_et^d\to \mathrm{Ext}^1_{\mathrm{BKF}^\circ_{\mathrm{rig}}}(A_\inf,A_\inf\{d\})
$$
is still surjective. Moreover, its kernel is given by $\Q_pt^d$ because $\Gamma(a,0)$ is trivial if and only if $a\in \Q_pt^d$.
The fundamental exact sequence implies therefore
$$
\mathrm{Ext}^1_{\mathrm{BKF}^\circ_{\mathrm{rig}}}(A_\inf,A_\inf\{d\})\cong B_et^d/\Q_pt^d\cong B_{\dR}/{t^{d}B^+_\dR}.
$$ 
\end{proof}

Let 
$$
K:=\mathrm{Ker}(\mathrm{Ext}^1_{\mathrm{BKF}^\circ}(A_\inf,A_{\inf}\{d\})\to \mathrm{Ext}^1_{X_{\mathrm{FF}}}(\mathcal{O}_{X_{\mathrm{FF}}},\mathcal{O}_{X_{\mathrm{FF}}}(d)))
$$
be the kernel of the natural map (in terms of modifications it sends $(\mathcal{F},\mathcal{F}^\prime,\beta)$ to $\mathcal{F}^\prime$). 
Analyzing the proof of \Cref{lemma: extensions_of_bkf_modules_of_rank_one} we see that
$$
K\cong B_et^d/(B^+_\crys)^{\varphi=p^d}.
$$ 
Thus imposing the condition of a rigidification enlarges this group to
$$
B_et^d/\Q_pt^d\cong B_{\dR}/t^dB^+_\dR.
$$

For more general $\mathrm{Ext}$-groups we can prove the following result (which implies \Cref{lemma: extensions_of_bkf_modules_of_rank_one}).
For simplicity, we denote by 
$$
X^\circ_{\mathrm{FF}}:=X_{\mathrm{FF}}\setminus\{\infty\}=\mathrm{Spec}(B_e)
$$
the punctured Fargues-Fontaine curve.

\begin{lemma}
\label{lemma:ext-groups-between-bkf-modules}
Let $(M_1,\varphi_{M_1}),(M_2,\varphi_{M_2})\in \mathrm{BKF}^0_{\mathrm{rig}}(M_1,M_2)$ be rigidified Breuil-Kisin-Fargues modules up to isogeny with associated modifications 
$$
(\mathcal{F}_i,\mathcal{F}_i^\prime,\beta_i,\alpha_i), i=1,2,
$$ 
of vector bundles on the Fargues-Fontaine curve as in \Cref{theorem: rigidified_bkf_modules_versus_modifications_of_vector_bundles}. Then there exist a natural surjection
$$
H^0(X_{\mathrm{FF}}^\circ,\mathcal{F}_{1|X^\circ_{\mathrm{FF}}}^{\vee}\otimes_{\mathcal{O}_{X^\circ_{\mathrm{FF}}}}\mathcal{F}_{2|X^\circ_{\mathrm{FF}}}^{\prime})\twoheadrightarrow \mathrm{Ext}^1_{\mathrm{BKF}^\circ_{\mathrm{rig}}}(M_1,M_2). 
$$
If $\mathcal{F}_i^\prime, i=1,2$, are semistable, then the kernel of this surjection is naturally isomorphic to
$$
H^0(X_{\mathrm{FF}},\mathcal{F}_1^\vee\otimes_{\mathcal{O}_{X_{\mathrm{FF}}}}\mathcal{F}_2).
$$
\end{lemma}
\begin{proof}
We first construct a natural morphism
$$
\Gamma\colon H^0(X_{\mathrm{FF}}^\circ,\mathcal{F}_{1|X^\circ_{\mathrm{FF}}}^{\prime\vee}\otimes_{\mathcal{O}_{X^\circ_{\mathrm{FF}}}}\mathcal{F}_{2|X^\circ_{\mathrm{FF}}}^{\prime})\to \mathrm{Ext}^1_{\mathrm{BKF}^\circ_{\mathrm{rig}}}(M_1,M_2).
$$  
Let 
$$
b\in H^0(X_{\mathrm{FF}}^\circ,\mathcal{F}_{1|X^\circ_{\mathrm{FF}}}^{\vee}\otimes_{\mathcal{O}_{X^\circ_{\mathrm{FF}}}}\mathcal{F}_{2|X^\circ_{\mathrm{FF}}}^{\prime})\cong \mathrm{Hom}_{X_\mathrm{FF}^\circ}(\mathcal{F}_{1|X^\circ_{\mathrm{FF}}},\mathcal{F}_{2|X^\circ_{\mathrm{FF}}}^\prime)
$$ 
be an element. Then we set $\Gamma(b)$ to be the extension
$$
\xymatrix{
0 \ar[r] & \mathcal{F}_{2} \ar[r]\ar@{-->}[d]^{\beta_2} & \mathcal{F}_2\oplus \mathcal{F}_1 \ar[r]\ar@{-->}[d]^{\begin{pmatrix} \beta_2 & b \\ 0 &\beta_1 \end{pmatrix}} & \mathcal{F}_1\ar[r]\ar@{-->}[d]^{\beta_1} & 0 \\
0 \ar[r] & \mathcal{F}_2^\prime \ar[r] & \mathcal{F}^\prime_2\oplus \mathcal{F}^\prime_1 \ar[r] & \mathcal{F}_1^\prime\ar[r] & 0 
}
$$
with rigidification
$$
\begin{pmatrix}
  \alpha_2 & 0 \\
 0 & \alpha_1
\end{pmatrix}\colon \mathcal{F}^\prime_2\oplus \mathcal{F}^\prime_1\cong \mathrm{gr}^\bullet(\mathrm{HN}(\mathcal{F}^\prime_2\oplus \mathcal{F}^\prime_1))=\mathrm{gr}^\bullet(\mathrm{HN}(\mathcal{F}^\prime_2))\oplus \mathrm{gr}^\bullet(\mathrm{HN}(\mathcal{F}^\prime_1)).
$$
We claim that $\Gamma$ is surjective with kernel, if $\mathcal{F}_1^\prime,\mathcal{F}_2^\prime$ are semistable, given by 
$$
H^0(X_{\mathrm{FF}},\mathcal{F}_1^\vee\otimes_{\mathcal{O}_{X_{\mathrm{FF}}}}\mathcal{F}_2)\cong \mathrm{Hom}_{X_\mathrm{FF}}(\mathcal{F}_{1},\mathcal{F}_{2}),
$$
embedded into $H^0(X_{\mathrm{FF}}^\circ,\mathcal{F}_{1|X^\circ_{\mathrm{FF}}}^{\vee}\otimes_{\mathcal{O}_{X^\circ_{\mathrm{FF}}}}\mathcal{F}_{2|X^\circ_{\mathrm{FF}}}^{\prime})$ via $c\mapsto \beta_2\circ c_{|X^\circ_{\mathrm{FF}}}$.
Let 
$$
(\mathcal{F},\mathcal{F^\prime},\beta,\alpha)
$$
be an extension of $(\mathcal{F}_1,\mathcal{F}_1^\prime,\beta_1,\alpha_1)$ by $(\mathcal{F}_2,\mathcal{F}_2^\prime,\beta_2,\alpha_2)$. Then 
$$
\mathcal{F}\cong \mathcal{F}_1\oplus\mathcal{F}_2
$$
because $H^1(X_{\mathrm{FF}},\mathcal{O}_{X_\mathrm{FF}})=0.$
Moreover, as the extension must be compatible with the rigidification the extension
$$
0\to \mathcal{F}_2^\prime\to \mathcal{F}^\prime\to \mathcal{F}^\prime_1\to 0
$$
splits, i.e., $\mathcal{F}^\prime\cong \mathcal{F}^\prime_1\oplus \mathcal{F}^\prime_2$. We moreover see that the rigidification $\alpha$ must be given by a matrix
$$
\alpha=\begin{pmatrix}
  \alpha_2 & a\\
  0 & \alpha_1
\end{pmatrix}
$$
with $a\in \mathrm{Hom}_{X_{\mathrm{FF}}}(\mathcal{F}^\prime_1,\mathrm{gr}(\mathrm{HN}(\mathcal{F}^\prime_2)))$. Moreover, $a$ must map to zero on the associated graded of the Harder-Narasimhan filtration
$$
\mathrm{gr}(a)\colon \mathrm{gr}(\mathrm{HN}(\mathcal{F}^\prime_1))\to \mathrm{gr}(\mathrm{HN}(\mathcal{F}^\prime_2))
$$ 
as $\alpha$ must reduce to the identity on these graded pieces.
Therefore there exists an $a^\prime\in \mathrm{Hom}_{X_{\mathrm{FF}}}(\mathcal{F}^\prime_1,\mathcal{F}^\prime_2)$ such that
$$
\alpha_2\circ a^\prime = \mathrm{gr}(a^\prime)\circ \alpha_1 + a.
$$ 
Then the extension
$$
(\mathcal{F},\mathcal{F^\prime},\beta,\alpha)
$$
is isomorphic to
$$
(\mathcal{F},\mathcal{F^\prime},
\begin{pmatrix}
  1 & a^\prime \\
  0 & 1
\end{pmatrix}\circ
\beta,
\begin{pmatrix}
  \alpha_2 & 0 \\
  0 & \alpha_1
\end{pmatrix}
)
$$
and thus this extension lies in the image of $\Gamma$. Now assume that for 
$$
b\in H^0(X_{\mathrm{FF}}^\circ,\mathcal{F}_{1|X^\circ_{\mathrm{FF}}}^{\vee}\otimes_{\mathcal{O}_{X^\circ_{\mathrm{FF}}}}\mathcal{F}_{2|X^\circ_{\mathrm{FF}}}^{\prime})
$$ 
the extension 
$$
\Gamma(b)
$$
is trivial. This means that there exists two isomorphisms 
$$
\begin{pmatrix}
  1 & c_1 \\
 0 & 1
\end{pmatrix}
\colon \mathcal{F}_1\oplus \mathcal{F}_2\cong \mathcal{F}_1\oplus\mathcal{F}_2,
$$
with $c_1\in \mathrm{Hom}_{X_{\mathrm{FF}}}(\mathcal{F}_1,\mathcal{F}_2)$, and
$$
\begin{pmatrix}
1 & c_2\\
0 & 1  
\end{pmatrix}\colon \mathcal{F}^\prime_1\oplus \mathcal{F}^\prime_2\cong \mathcal{F}^\prime_1\oplus\mathcal{F}^\prime_2,
$$
with $c_2\in \mathrm{Hom}_{X_{\mathrm{FF}}}(\mathcal{F}_1^\prime,\mathcal{F}_2^\prime)$ inducing zero on the graded pieces of the Harder-Narasimhan filtration, such that
$$
\begin{pmatrix}
  \beta_2 & b \\
  0 & \beta_1
\end{pmatrix}
\begin{pmatrix}
  1 & c_1\\
  0 & 1
\end{pmatrix}=
\begin{pmatrix}
  1 & c_2 \\
  0 & 1
\end{pmatrix}
\begin{pmatrix}
  \beta_2 & 0 \\
  0 & \beta_1
\end{pmatrix}.
$$
If $\mathcal{F}^\prime_i,i=1,2,$ are semistable, then $c_2=0$ must be zero and thus 
$$
b+\beta_2\circ c_1=0,
$$
i.e., $b=-\beta_2\circ c_1$. This implies therefore the last statement in the lemma.
\end{proof}

We want to deduce that the category $\mathrm{BKF}^\circ_{\mathrm{rig}}$ of rigidified Breuil-Kisin-Fargues modules up to isogeny is of homological dimension $1$.

\begin{lemma}
\label{lemma:homological-dimension}
The category $\mathrm{BKF}^\circ_{\mathrm{rig}}$ of rigidified Breuil-Kisin-Fargues modules up to isogeny is of homological dimension $1$, i.e., for every $(M,\varphi_M), (N,\varphi_N)\in \mathrm{BKF}^\circ_{\mathrm{rig}}$ the Ext-group (in the sense of Yoneda or equivalently as spaces of homomorphisms in the derived category) vanishes for $i\geq 2$, i.e., 
$$
\mathrm{Ext}^i_{\mathrm{BKF}^\circ_{\mathrm{rig}}}(M,N)=0.
$$
\end{lemma}
\begin{proof}
\Cref{lemma:ext-groups-between-bkf-modules} implies that the functor 
$
\mathrm{Ext}^1_{\mathrm{BKF}^\circ_{\mathrm{rig}}}(M,-)
$  
preserves surjections of rigidified Breuil-Kisin-Fargues modules. Namely, for a surjection
$$
(N_1,\varphi_{N_1})\twoheadrightarrow (N_2,\varphi_{N_2})
$$
of rigidified Breuil-Kisin-Fargues modules the associated surjection
$$
\mathcal{F}_1^\prime\twoheadrightarrow \mathcal{F}_2^\prime
$$
of vector bundles on the Fargues-Fontaine curve must be split due to the preservation of the given rigidifications.
For an abelian category with enough injectives (or projectives) this would finish the proof by embedding an object into an injective object and using the associated long exact sequence. However, the category $\mathrm{BKF}^\circ_{\mathrm{rig}}$ does not contain enough injectives (or projectives), hence we have to work a bit more.
For this let $\mathcal{A}$ be an essentially small\footnote{This means that the isomorphism classes in $\mathcal{A}$ form a set.} abelian category such that 
$$
\mathrm{Ext^1}_{\mathcal{A}}(A,-)
$$
preserves surjections for every $A\in \mathcal{A}$.
Consider the (fully faithful and exact) embedding
$$
\mathcal{A}\to \mathrm{Ind}(\mathcal{A})
$$
of $\mathcal{A}$ into the category $\mathrm{Ind}(\mathcal{A})$ of its Ind-objects. The category $\mathrm{Ind}(\mathcal{A})$ need not have to have enough injectives in general (\cite[Corollary 15.1.3.]{kashiwara_schapira_categories_and_sheaves}), but it has if $\mathcal{A}$ is essentially small which we assumed (in fact it is then Grothendieck abelian, cf.\ \cite[Theorem 8.6.5.(vi)]{kashiwara_schapira_categories_and_sheaves}). Hence, we can conclude that (cf.\ \cite[Corollary 15.3.9.]{kashiwara_schapira_categories_and_sheaves})
$$
\varinjlim\limits_{A^\prime\to B } \mathrm{Ext}^k_{\mathcal{A}}(A,A^\prime)\cong \mathrm{Ext}^k_{\mathrm{Ind}(\mathcal{A})}(A,B)
$$
for $A\in \mathcal{A}, B\in \mathrm{Ind}(\mathcal{A})$ where the colimit is running over all $A^\prime\in \mathcal{A}$ with a morphism to $B$.
Now let $f\colon B^\prime\to B^{\prime\prime}$ be a surjection of Ind-objects. Then $f$ can be written as a filtered colimit of surjections $A^\prime\to A^{\prime\prime}$ with $A^\prime,A^{\prime\prime}\in \mathcal{A}$ (cf.\ \cite[Proposition 8.6.6.]{kashiwara_schapira_categories_and_sheaves}). Using the above formular for the $\mathrm{Ext}$-groups and our assumption on $\mathcal{A}$ we can conclude that for $A\in \mathcal{A}$ the functor
$$
\mathrm{Ext}^1_{\mathrm{Ind}(\mathcal{A})}(A,-)
$$
preserves surjections. Using embeddings into injectives (of which there are enough in $\mathrm{Ind}(\mathcal{A})$ by our assumption on essential smallness) we can conclude that 
$$
\mathrm{Ext}^i_{\mathrm{Ind}(\mathcal{A}}(A,-)=0
$$
for $A\in \mathcal{A}$ and $i\geq 2$.
Using again the above formula for the $\mathrm{Ext}$-groups we can conclude
$$
\mathrm{Ext}^i_{\mathcal{A}}(A,A^\prime)=0
$$
for $A,A^\prime\in \mathcal{A}$ and $i\geq 2$.

Applying these considerations to the category $\mathcal{A}=\mathrm{BKF}^\circ_{\mathrm{rig}}$ proves the lemma if we can show that the category $\mathrm{BKF}^\circ_{\mathrm{rig}}$ is essentially small. But by definition Breuil-Kisin-Fargues modules are finitely presented and the isomorphism classes of finitely presented $A_\inf$-modules form a set. As the possibilities for adding a Frobenius or a rigidification form a set, we can conclude that the isomorphism classes of Breuil-Kisin-Fargues modules form a set as required.
\end{proof}

\section{CM Breuil-Kisin-Fargues modules}
\label{section: cm_breuil_kisin_fargues_modules}

In this section we want to apply the formal CM theory of \Cref{section: formal_cm_theory} to the case 
$$
\mathcal{T}:=\mathrm{BKF}^\circ_{\mathrm{rig}}
$$ 
of rigidified Breuil-Kisin-Fargues modules up to isogeny (cf.\ \Cref{definition: rigidified_bkf_modules_up_to_isogeny}).

Using Fargues' theorem \Cref{theorem: equivalent_descriptions_of_bkf_modules} the classification of CM Breuil-Kisin-Fargues modules, i.e., CM objects in the Tannakian category $\mathcal{T}$ (cf.\ \Cref{definition:cm-object}), is actually very simple - they are uniquely determined by their CM type $(E,\Phi)$ (cf.\ \Cref{lemma:integral-cm-bkf-modules} and \Cref{lemma-rigidifications-of-cm-bkf-modules}).

Actually, we can prove a stronger integral statement. For this let $E/\Q_p$ be a commutative semisimple algebra and let $\mathcal{O}$ be an order in $E$.
\begin{definition}
\label{definition-integral-cm-bkf-module}
 A finite free Breuil-Kisin-Fargues module with CM by $\mathcal{O}$ will mean a finite free Breuil-Kisin-Fargues module $(M,\varphi_M)$ together with an injection $\mathcal{O}\to \mathrm{End}_{\mathrm{BKF}}((M,\varphi_M))$, such that $\mathrm{rk}(M)=\mathrm{rk}_{\Z_p}(\mathcal{O})$.  
\end{definition}

\begin{lemma}
\label{lemma:integral-cm-bkf-modules}
Let $E$ be a commutative, semisimple, finite-dimensional algebra over $\Q_p$ and let $\mathcal{O}\subseteq E$ be an order in $E$. Then there is a natural equivalence of categories between finite free Breuil-Kisin-Fargues module with CM by $\mathcal{O}$
and pairs $(T,\Phi)$ where $T$ is a faithful $\mathcal{O}$-module, finite free over $\Z_p$ of rank $\rk_{\Z_p}(\mathcal{O})$, and functions $\Phi\colon \mathrm{Hom}_{\Q_p}(E,C)\to \Z$.
In particular, if $\mathcal{O}=\mathcal{O}_E$ is the maximal order, then there is a bijection of finite free Breuil-Kisin-Fargues modules with CM by $\mathcal{O}_E$ (up to isomorphism) and ``types'' $\Phi\colon \mathrm{Hom}_{\Q_p}(E,C)\to \Z$.
\end{lemma}
\begin{proof}
By \Cref{theorem: equivalent_descriptions_of_bkf_modules} finite free Breuil-Kisin-Fargues modules with CM by $\mathcal{O}$ are equivalent to pairs $(T,\Xi)$ with $T$ as in the statement of this lemma and $\Xi\subseteq T\otimes_{\Z_p}B_{\dR}$ a $B^+_{\dR}$-sublattice, stable under $\mathcal{O}\otimes_{\Z_p}B^+_{\dR}$.
But $B^+_{\dR}$ contains $\overline{\Q}_p$, hence
$$
\mathcal{O}\otimes_{\Z_p}B_\dR\cong (E\otimes_{\Q_p}{\overline{\Q}_p})\otimes_{\overline{\Q}_p}B_\dR\cong \prod\limits_{\mathrm{Hom}_{\Q_p}(E,\overline{\Q}_p)}B_\dR
$$  
and a $\mathcal{O}\otimes_{\Z_p}B^+_{\dR}$-stable sublattice will be uniquely determined by the valuation in each factor, i.e., by a function
$$
\Phi\colon \mathrm{Hom}_{\Q_p}(E,\overline{\Q}_p)=\mathrm{Hom}_{\Q_p}(E,C)\to\Z.
$$
If $\mathcal{O}=\mathcal{O}_E$ is the maximal order, then moreover every faithful $\mathcal{O}_E$-module $T$ which is finite free of rank $d$ over $\Z_p$ must be isomorphic to $\mathcal{O}_E$. This finishes the proof.
\end{proof}

We remark that in general there are non-trivial examples of $\mathcal{O}$-modules $T$ satisfying the hypothesis in \Cref{lemma:integral-cm-bkf-modules}.
We thank Bhargav Bhatt and Sebastian Posur for discussions about this point.
In general, one can take $\mathcal{O}\neq \mathcal{O}_E$ and $T:=\mathcal{O}_E$. But there exist also less pathological examples. For example, let $E=\Q_p(p^{1/4})$ and set $\mathcal{O}:=\Z_p[p^{2/4},p^{3/4}]$. Then $\mathcal{O}/p\cong \F_p[t^2,t^3]/t^4\cong\F_p[x,y]/(x^2,xy,y^2)$ is not Gorentstein. In particular, $\mathcal{O}$ is not Gorenstein as well (the dualizing complex commutes with (derived) base change). However, it is still Cohen-Macaulay. Hence, the dualizing module $T:=\omega_{\mathcal{O}}\cong \mathrm{Hom}_{\Z_p}(\mathcal{O},\Z_p)$ (cf.\ \cite[Tag 0A7B, Tag 0AWS]{stacks_project}) for $\mathcal{O}$ yields an example. It is finite free over $\Z_p$ of rank $\mathrm{rk}_{\Z_p}\mathcal{O}$ and $\mathrm{End}_{\mathcal{O}}(\omega_{\mathcal{O}})\cong \mathcal{O}$. As $\mathcal{O}$ is not Gorenstein, $\omega_{\mathcal{O}}\neq \mathcal{O}$.   

We now analyse rigidifications of Breuil-Kisin-Fargues modules with CM.

\begin{lemma}
\label{lemma-rigidifications-of-cm-bkf-modules}
Let $E$ be a commutative semisimple algebra over $\Q_p$ of dimension $d$.
Let $(M,\varphi_M)$ be a finite free Breuil-Kisin-Fargues module of rank $\mathrm{rk}(M)=d$ with an injection $E\hookrightarrow \mathrm{End}_{\mathrm{BKF}^\circ}((M,\varphi_M))$. Then there exists a unique rigidification
$$
\alpha\colon M\otimes_{A_\inf}B^+_\crys\cong (M\otimes_{A\inf} W(k))\otimes_{W(k)}B^+_\crys
$$  
which is preserved by $E$, i.e., $E$-linear.
\end{lemma}
\begin{proof}
We use \Cref{theorem: rigidified_bkf_modules_versus_modifications_of_vector_bundles} (respectively the remark following it) to argue with the modification $(\mathcal{F},\mathcal{F}^\prime,\beta)$ of vector bundles on the Fargues-Fontaine curve associated with $(M,\varphi_M)$. The algebra $E$ acts by assumption on the vector bundle $\mathcal{F}^\prime$ of rank $d$ and we must produce a unique $E$-linear isomorphism
$$
\mathcal{F}^\prime\cong \mathrm{gr}^\bullet(\mathrm{HN}(\mathcal{F}^\prime)).
$$   
Decomposing $E$ (and then $\mathcal{F}^\prime$ accordingly) into factors reduces to the case that $E$ is a field. Then we claim that $\mathcal{F}^\prime$ must be semistable. Indeed, each subbundle $\mathcal{E}^\lambda\subseteq\mathcal{F}^\prime$ in the Harder-Narasimhan filtration of $\mathcal{F}^\prime$ must be stable under $E$. Let $K$ be the function field of $X_{\mathrm{FF}}$. Then $E\otimes_{\Q_p}K$ is again a field (because $X_{\mathrm{FF}}\otimes_{\Q_p}E$ is again integral, cf.\ \cite[Th\'eor\`eme 6.5.2.2.)]{fargues_fontaine_courbes_et_fibres_vectoriels_en_theorie_de_hodge_p_adique}) and passing to the generic point $\eta\in X_{\mathrm{FF}}$ yields a $E\otimes_{\Q_p}K$-stable flag in the $1$-dimensional $E\otimes_{\Q_p}K$-vector space $\mathcal{F}^\prime_\eta$. Hence, this flag is trivial and thus $\mathcal{F}^\prime$ semistable.
This implies finally that there exists a unique $E$-linear rigidification for $\mathcal{F}^\prime$, namely the identity of $\mathcal{F}^\prime$.
\end{proof}

In other words, we can write down all rigidified Breuil-Kisin-Fargues modules up to isogeny with CM by a finite-dimensional commutative, semisimple $\Q_p$-algebra $E$ in terms of pairs $(V,\Xi)$. Namely, $V$ must be isomorphic (as an $E$-module) to $E$ and $\Xi\subseteq E\otimes_{\Q_p}B_\dR$ can be constructed explicitly from the type
$$
\Phi\colon \mathrm{Hom}_{\Q_p}(E,C)\to \Z.
$$ 
We now want to write down the Breuil-Kisin-Fargues modules $(M,\varphi_M)$ corresponding to a pair $(E,\Phi)$. This will require more work.

We note that for a finite free Breuil-Kisin-Fargues module $(M,\varphi_{M_\Phi})$ as in \Cref{lemma:integral-cm-bkf-modules} the function $\Phi\colon \mathrm{Hom}_{\Q_p}(E,C)\to \Z$ in \Cref{lemma:integral-cm-bkf-modules} is precisely the type of the rigidified Breuil-Kisin-Fargues modules up to isogeny $(\Q_p\otimes_{\Z_p}M,\varphi_{M_\Phi})\in \mathrm{BKF}^\circ_{\mathrm{rig}}$ with respect to the filtered fiber functor
$$
\omega_{\acute{e}t}\otimes_{\Q_p}C\colon \mathrm{BKF}^\circ_{\mathrm{rig}}\to \mathrm{Vec}_C
$$
(cf.\ \Cref{lemma: filtered_fiber_functor_on_rigidified_bkf_modules} and \Cref{definition:type-of-a-cm-object}). Namely, this follows from the concrete description of the filtration on $\omega_{\acute{e}t}\otimes_{\Q_p}C$ (cf.\ the discussion after \Cref{definition:filtration-on-etale-realization} and the proof of \Cref{theorem-rigidified-bkf-modules-with-cm}).

In general a Breuil-Kisin-Fargues module with CM by a commutative semisimple $\Q_p$-algebra $E$ will decompose according to the factors of $E$. In particular, we may focus on the case where $E$ is a field.
Hence, we fix a finite extension $E$ of $\Q_p$ and denote by $E_0\subseteq E$ its maximal unramified subextension. We fix a uniformizer $\pi\in E$. 

Let $\overline{\F}_p\subseteq k$ be the algebraic closure of $\F_p$. By formal \'etaleness of $\overline{\F}_p$ over $\F_p$ there exists a unique lifting  $\overline{\F}_p\hookrightarrow \mathcal{O}_C/p$ of the embedding $\overline{\F}_p\to k$. Concretely, 
$$
\overline{\F}_p=\bigcup\limits_{n\geq 0}(\mathcal{O}_{C}/p)^{\varphi^n=1}.
$$
Taking the inverse limit over Frobenius yields a canonical embedding
$$
\overline{\F}_p\hookrightarrow \mathcal{O}_C^\flat.
$$
The algebraic closure of $\Q_p$ in $\Q_p\otimes_{\Z_p}A_\inf$ is contained in
$$
\breve{\Q}_p\cong \Q_p\otimes_{\Z_p}W(\overline{\F}_p)\subseteq\Q_p\otimes_{\Z_p}A_\inf
$$
where $\breve{\Q}_p$ denotes the completion of the maximal unramified extension $\Q_p^{\mathrm{un}}$ of $\Q_p$. In particular, the algebraic closure of $\Q_p$ in $\Q_p\otimes_{\Z_p}A_\inf$ is given by $\Q_p^{\mathrm{un}}$.
For an embedding of $\iota\colon E_0\hookrightarrow \Q_p^{\mathrm{un}}$ we define
$$
A_{\inf,\mathcal{O}_E,\iota}:=\mathcal{O}_E\otimes_{\mathcal{O}_{E_0},\iota}A_\inf,
$$
the ring of ``ramified Witt vectors'' (cf.\ \cite[Section 1.2.]{fargues_fontaine_courbes_et_fibres_vectoriels_en_theorie_de_hodge_p_adique}). Its elements are formal power series
$$
\sum\limits_{i=0}^\infty [x_i]\pi^i
$$
with $x_i\in \mathcal{O}_C^\flat$.
We set
$$
A_{\inf,E,\iota}:=\Q_p\otimes_{\Z_p}A_{\inf,\mathcal{O}_E,\iota}=E\otimes_{\mathcal{O}_E}A_{\inf,\mathcal{O}_E,\iota}.
$$

\begin{lemma}
\label{lemma-kernel-after-fixing-embedding}
Let $\tau\colon E\to C$ be an embedding and let $\tau_0\colon E_0\to C$ be its restriction to $E_0$. Then the kernel of the homomorphism
$$
\theta_{\tau}\colon \mathcal{O}_E\otimes_{\Z_p}A_\inf\to C,\ e\otimes x\mapsto \tau(e)\theta(x) 
$$ 
is principal, generated by a non-zero divisor $\xi_\tau$.
In the decomposition
$$
\mathcal{O}_E\otimes_{\Z_p}A_\inf\cong \mathcal{O}_E\otimes_{\mathcal{O}_{E_0}}(\mathcal{O}_{E_0}\otimes_{\Z_p}A_\inf)\cong \prod\limits_{\iota\colon E_0\hookrightarrow C} A_{\inf,\mathcal{O}_E,\iota}
$$
the element $\xi_\tau$ can be chosen to be
$$
(1,\ldots,\pi-[\tau(\pi)^\flat],\ldots,1)
$$
with $\pi-[\tau(\pi)^\flat]$ placed in the component $\tau_0$.
Here $\tau(\pi)^\flat=(\tau(\pi),\tau(\pi)^{1/p},\ldots)$ denotes a $p$-power compatible systems of $p$-power roots of $\tau(\pi)\in C$.
\end{lemma}
\begin{proof}
The morphism $\theta_\tau\colon \Spec(C)\to \Spec(E\otimes_{\Z_p}A_\inf)$ must factor through one component and this component must be the one corresponding to the factor $A_{\inf,E,\tau_0}$ because $\theta_\tau$ factors over $A_{\inf,E,\tau_0}$. Then the statement is well-known (cf.\ \cite[Lemme 2.1.9.]{fargues_fontaine_courbes_et_fibres_vectoriels_en_theorie_de_hodge_p_adique}).
\end{proof}

Being a non-zero divisor the element $\xi_\tau$ in \Cref{lemma-kernel-after-fixing-embedding} is unique up to a unit.
Moreover, 
$$
\prod\limits_{\tau\in \mathrm{Hom}_{\Q_p}(E,C)}\xi_\tau=u\xi
$$
with $u\in \mathcal{O}_E\otimes_{\Z_p}A_\inf$ a unit. Indeed, tensoring the exact sequence
$$
0\to A_\inf\xrightarrow{\xi}A_\inf\xrightarrow{\theta}\mathcal{O}_C\to 0
$$
with $\mathcal{O}_E$ yields the sequence
$$
0\to \mathcal{O}_E\otimes_{\Z_p}A_\inf\xrightarrow{\xi}\mathcal{O}_E\otimes_{\Z_p}A_\inf \xrightarrow{\prod\theta_{\tau}} \prod\limits_{\tau\colon \mathcal{O}_E\to\mathcal{O}_C}\mathcal{O}_C\cong \mathcal{O}_E\otimes_{\Z_p}\mathcal{O}_C\to 0
$$
which implies that the vanishing locus $\xi$ and $\prod\limits_{\tau}\xi_\tau$ generate the same ideal. As both elements are non-zero divisors, they differ by a unit.

\begin{definition}
\label{definition-concrete-bkf-module-with-cm}
Let $\Phi\colon \Hom_{\Q_p}(E,C)\to \Z$ be a type. Then we define the finite free Breuil-Kisin-Fargues module with CM by $\mathcal{O}_E$ as
$M_\Phi:=\mathcal{O}_E\otimes_{\Z_p}A_\inf$ with Frobenius
$$
\varphi_{M_\Phi}:=\tilde{\xi}_{\Phi}\varphi
$$
where
$$
\tilde{\xi}_{\Phi}:=\prod\limits_{\tau\in \mathrm{Hom}_{\Q_p}(E,C)}\varphi(\xi_\tau)^{\Phi(\tau)}
$$
with $\xi_\tau$ as in \Cref{lemma-kernel-after-fixing-embedding} and $\varphi=\mathrm{Id}\otimes \varphi_{A_\inf}\colon \mathcal{O}_E\otimes_{\Z_p}A_\inf\to \mathcal{O}_E\otimes_{\Z_p}A_\inf$.
\end{definition}

First observe that
$$
\tilde{\xi}_\tau:=\varphi(\xi_\tau)
$$
is a generator of the morphism
$$
\tilde{\theta}_\tau\colon \mathcal{O}_E\otimes_{\Z_p}A_\inf\to C,\ e\otimes x\mapsto \tau(e)\tilde{\theta}(x)
$$
which extends the morphism $\tilde{\theta}\colon A_\inf\to C$ with kernel $\tilde{\xi}$. Hence  every $\tilde{\xi}_\tau$ is a unit in $\mathcal{O}_E\otimes_{\Z_p}A_\inf[\frac{1}{\tilde{\xi}}]$ and thus 
$$
\varphi_{M_\Phi}\colon \varphi^\ast(M_\Phi)[\frac{1}{\tilde{\xi}}]\to M_\Phi[\frac{1}{\tilde{\xi}}]
$$ is indeed an isomorphism of $A_\inf[\frac{1}{\tilde{\xi}}]$-modules.
Moreover, the multiplication by $\mathcal{O}_E$ induces a multiplication on $M_\Phi$ and thus $M_\Phi$ is a finite free Breuil-Kisin-Fargues modules with CM by $\mathcal{O}_E$. To determine the isomorphism class of $M_\Phi$ it thus suffices (cf.\ \Cref{lemma:integral-cm-bkf-modules}) to compute the type of $M_\Phi$ (which of course will turn out to be $\Phi$).
We check independently of \Cref{lemma:integral-cm-bkf-modules} that the (isomorphism class) of the module $(M_\Phi,\varphi_{M_\Phi})$ is independent of the choice of the elements $\xi_\tau$.

We recall the following lemma. If $d$ is degree of $E_0$ over $\Q_p$, then we set 
$$
\varphi_{E_0}:=\varphi^d.
$$ 
\begin{lemma}
\label{lemma-solutions-for-varphi-equations}
Fix $\iota\in \mathrm{Hom}_{\Q_p}(E_0,C)$. Then for every $x=\sum\limits_{i\geq 0}^{\infty}[x_i]\pi^i\in A_{\inf,\mathcal{O}_E,\iota}$ with $x_0\neq 0$ the $\mathcal{O}_E$-module 
$$
P_x:=\{y\in A_{\inf,\mathcal{O}_E,\iota} |\ \varphi_{E_0}(y)=xy\}
$$
is free of rank $1$. Furthermore, if $x\in A_{\inf,\mathcal{O}_E,\iota}$ is a unit, then a generator of $P_x$ is a unit in $A_{\inf,\mathcal{O}_E,\iota}$.
\end{lemma}
\begin{proof}
After tensoring with $\Q_p$ the first assertion is proven in the proof of \cite[Proposition 6.2.10.]{fargues_fontaine_courbes_et_fibres_vectoriels_en_theorie_de_hodge_p_adique}. But $P_x$ is $\varpi$-torsion free and hence free over $\mathcal{O}_E$ of rank $1$.  The second part follows from the proof of \cite[Proposition 6.2.10.]{fargues_fontaine_courbes_et_fibres_vectoriels_en_theorie_de_hodge_p_adique}.
\end{proof}

For example, if $E=\Q_p$ and $x=\tilde{\xi}$, then $y=\mu$ spans the space $P_{\tilde{\xi}}$ as
$$
\varphi(\mu)=[\varepsilon^p]-1=\frac{[\varepsilon^p]-1}{[\varepsilon]-1}([\varepsilon^p]-1)=\tilde{\xi}\mu.
$$ 
In this case, the space $P_{\tilde{\xi}}$ does not contain a unit in $A_\inf$.

We strengthen \Cref{lemma-solutions-for-varphi-equations} a bit to handle the non-connected ring $\mathcal{O}_E\otimes_{\Z_p}A_\inf$ as well.

\begin{lemma}
\label{lemma-disconnected-case-for-varphi-solutions}
Let $x\in \mathcal{O}_E\otimes_{\Z_p}A_\inf$ such that $x$ maps to a unit in $\mathcal{O}_E/p\otimes_{\F_p}C^\flat$. Then the space
$$
P_x:=\{y\in \mathcal{O}_E\otimes_{\Z_p}A_\inf\ |\ \varphi(y)=xy\}
$$
is free of rank $1$ over $\mathcal{O}_E$.
Here $\varphi$ denotes the Frobenius $\Id_{\mathcal{O}_E}\otimes \varphi_{A_\inf}$ on $\mathcal{O}_E\otimes_{\Z_p}A_\inf$. If $x$ is already a unit in $\mathcal{O}_E\otimes_{\Z_p}A_\inf$, then $P_x$ contains a unit in $\mathcal{O}_E,\otimes_{\Z_p}A_\inf$.
\end{lemma}
\begin{proof}
Write 
$$
\mathcal{O}_E\otimes_{\Z_p}A_\inf\cong \prod\limits_{\iota\in \mathrm{Hom}_{\Q_p}(E_0,C)} A_{\inf,\mathcal{O}_E,\iota}
$$
and thus the element $x=(x_\iota)_{\iota\in \mathrm{Hom}_{\Q_p}(E_0,C)}$ accordingly. The assumption on $x$ implies that each $x_\iota$ satisfies the assumption in \Cref{lemma-solutions-for-varphi-equations}. If $x\in \mathcal{O}_E\otimes_{\Z_p}A_\inf$ is a unit, then as well each $x_\iota$ is a unit. Moreover, the Frobenius $\varphi$ permutes the factors cyclically. The power
$$
\varphi^d=\varphi_{\mathcal{O}_{E_0}}
$$
fixes every factor and induces, for a fixed $\iota$, the morphism $\varphi_{\mathcal{O}_{E_0}}$ on $A_{\inf,\mathcal{O}_E,\iota}$. Fix some $\iota_{0}$ and let $y_0\in A_{\inf,\mathcal{O}_E,\iota}$ be a generating solution of
$$
\varphi_{\mathcal{O}_E}(y_0)=x_{\iota_0} y_0
$$ 
($y_0$ is a unit if $x$ is a unit). As $\varphi$ permutes the factors, it is clear that we get a generating solution $y\in A_{\inf,\mathcal{O}_E,\iota}$ (which is a unit if $x$ is a unit) for the equation
$$
\varphi(y)=xy.
$$
Moreover, this $y$ must be unique up to multipliciation by $\mathcal{O}_E^\times$.
\end{proof}

\begin{lemma}
\label{different-choices-yield-same-bkf-module}
The Breuil-Kisin-Fargues module $(M_\Phi,\varphi_{M_\Phi})$ in \Cref{definition-concrete-bkf-module-with-cm} is up to isomorphism independent of the choice of the elements $\xi_\tau$ in \Cref{lemma-kernel-after-fixing-embedding}.
\end{lemma}
\begin{proof}
A different choice of the elements $\xi_\tau$ yields the Breuil-Kisin-Fargues module
$$
(M^\prime_\Phi,\varphi_{M^\prime_\Phi}):=(\mathcal{O}_E\otimes_{\Z_p}A_\inf,u\tilde{\xi}_\Phi\varphi)
$$
with $u\in (\mathcal{O}_E\otimes_{\Z_p}A_\inf)^\times$ a unit.
By \Cref{lemma-disconnected-case-for-varphi-solutions} we can find a unit
$$
y\in (\mathcal{O}_E\otimes_{\Z_p}A_\inf)^\times
$$
such that
$$
\varphi(y)=uy.
$$
Then multiplication by $y$ will define an ($\mathcal{O}_E$-linear) isomorphism
$$
(M^\prime_\Phi,\varphi_{M^\prime_\Phi})\to (M_{\Phi},\varphi_{M_\Phi}).
$$
  
\end{proof}

We can now describe all (rigidified) Breuil-Kisin-Fargues modules with CM.

\begin{theorem}
\label{theorem-rigidified-bkf-modules-with-cm}
Let $\Phi\colon \mathrm{Hom}_{\Q_p}(E,C)\to \Z$ be a type. Then the finite free Breuil-Kisin-Fargues modules
$$
(M_\Phi,\varphi_{M_\Phi})
$$ 
from \Cref{definition-concrete-bkf-module-with-cm} has type $\Phi$. Moreover, every finite free Breuil-Kisin-Fargues modules with CM by $\mathcal{O}_E$ with type $\Phi$ is isomorphic to $(M_\Phi,\varphi_{M_\Phi})$.
\end{theorem}
\begin{proof}
If we can show that $(M_\Phi,\varphi_{M_\Phi})$ has type $\Phi$, then the last assertion follows from \Cref{lemma:integral-cm-bkf-modules}.  
Fix as in \Cref{lemma-kernel-after-fixing-embedding} elements
$$
\xi_\tau\in \mathcal{O}_E\otimes_{\Z_p}A_\inf
$$
generating the kernel of
$$
\theta_r\colon \mathcal{O}_E\otimes_{\Z_p}A_\inf\to \mathcal{O}_C
$$
with $\tau$ running over all embeddings $\tau\colon E\to C$.
For simplicity we may (after multiplying one $\xi_\tau$ by a unit in $\mathcal{O}_E\otimes_{\Z_p}A_\inf$) assume that
$$
\xi=\prod\limits_{\tau}\xi_\tau
$$
(cf.\ the discussion after \Cref{lemma-kernel-after-fixing-embedding}).
Set
$$
\tilde{\xi}_\tau:=\varphi(\xi_\tau)\in \mathcal{O}_E\otimes_{\Z_p}A_\inf
$$
and let $\mu_\tau\in \mathcal{O}_E\otimes_{\Z_p}A_\inf$ be as in \Cref{lemma-disconnected-case-for-varphi-solutions} the unique (up to multiplication by $\mathcal{O}_E^\times$) generator for  elements satisfying the equation
$$
\varphi(\mu_\tau)=\tilde{\xi}_\tau\mu_\tau
$$
(if $E=\Q_p$ and $\tau\colon \Q_p\to C$ is the unique inclusion, then $\xi_\tau$ can be taken to be $\xi$, which implies $\mu_\tau=\mu$ (up to a unit)).
Note that $\mu_\tau$ is a unit in $E\otimes_{\Z_p}A_\inf[\frac{1}{\mu}]$. Indeed, the product
$$
\nu:=\prod\limits_{\tau}\mu_\tau
$$
satisfies
$$
\varphi(\nu)=\tilde{\xi}\nu
$$
as we assumed $\xi=\prod\limits_{\tau}\xi_\tau$.
Hence, by \Cref{lemma-solutions-for-varphi-equations} $\nu$ and $\mu$ differ by a scalar in $E$.
Applying $\varphi^{-1}$ to the defining equation of $\mu_\tau$ yields 
$$
\mu_\tau=\xi_\tau\varphi^{-1}(\mu_\tau)
$$ 
and we see that $\varphi^{-1}(\mu_\tau)$ is a unit in $(E\otimes_{\Z_p}A_\inf)\otimes_{A_\inf}B^+_\dR$ because the same is true for $\varphi^{-1}(\mu)$ (one checks $\theta(\varphi^{-1}(\mu))=\zeta_p-1\neq 0$).
We note that the type of $M_\Phi$ depends only on the Breuil-Kisin-Fargues modules up to isogeny $E\otimes_{\mathcal{O}_E}M_{\Phi}$.
Set 
$$
V:=(M_\Phi\otimes_{A_\inf}W(C^\flat)[\frac{1}{p}])^{\varphi_{M_\Phi}=1}
$$
Then we know that $V$ is a one-dimensional $E$-vector space. Recall that
$$
\varphi_{M_\Phi}=\tilde{\xi}_\Phi\varphi
$$
with 
$$
\tilde{\xi}_\Phi:=\prod\limits_{\tau}\tilde{\xi}_\tau^{\Phi(\tau)}.
$$
We can explicitly find a generator of $V$, namely the element
$$
\mu_\Phi:=\prod\limits_{\tau}\mu_{\tau}^{-\Phi(\tau)}\in E\otimes_{\Z_p}A_\inf[\frac{1}{\mu}].
$$
Indeed,
$$
\varphi(\mu_\Phi)=\tilde{\xi}_{\Phi}^{-1}\mu_\Phi,
$$
which implies
$$
\varphi_{M_{\Phi}}(\mu_\Phi)=\mu_\Phi.
$$
By definition, the Breuil-Kisin-Fargues modules (up to isogeny) $(E\otimes_{\mathcal{O}_E}M_\Phi,\varphi_{M_\Phi})$ corresponds to the pair
$$
(V,\Xi:=M_\Phi\otimes_{A_\inf}B^+_\dR\subseteq V\otimes_{\Q_p}B_\dR)
$$
in (the isogeny version) of \Cref{theorem: equivalent_descriptions_of_bkf_modules}. 
Now, the $E\otimes_{\Q_p}B^+_\dR$-lattice 
$$
V\otimes_{\Q_p}B^+_\dR\subseteq V\otimes_{\Q_p}B_\dR=E\otimes_{\Q_p}B_\dR
$$ 
is generated by $\mu_{\Phi}$. But $\mu_{\Phi}$ equals, up to a unit in $E\otimes_{\Q_p}B^+_\dR$, the inverse $\xi_{\Phi}^{-1}$ of the element
$$
\xi_\Phi:=\prod\limits_{\tau}\xi_\tau^{\Phi(\tau)}=\varphi^{-1}(\tilde{\xi}_\Phi)
$$
 in $E\otimes_{\Q_p}B^+_\dR$ while the $E\otimes_{\Q_p}B^+_{\dR}$-lattice $\Xi$ is generated by $1\in E\otimes_{\Q_p}B^+_\dR$. In other words, we find that
$$
\Xi=\xi_\Phi(V\otimes_{\Q_p}B^+_\dR)
$$
which implies that $(M_\Phi,\varphi_{M_\Phi})$ has type $\Phi$ by looking at the explicit decomposition
$$
E\otimes_{\Q_p}B_\dR\cong \prod\limits_{\tau\in \mathrm{Hom}_{\Q_p}(E,C)}B_\dR
$$
under which $\xi_\Phi$ maps to the element $(\xi^{\Phi(\tau)})_\tau$.
\end{proof}

We want to finish this chapter with a concrete description of the automorphisms
$$
\mathrm{Aut}^\otimes(\omega_{\acute{e}t,\mathrm{CM}})
$$
of the fiber functor, ``the \'etale realization'',
$$
\omega_{\acute{e}t,\mathrm{CM}}\colon \mathrm{BKF}^\circ_{\mathrm{rig},\mathrm{CM}}\to \mathrm{Vec}_{\Q_p}
$$
on Breuil-Kisin-Fargues modules with CM.

Recall the pro-torus 
$$
D_{\Q_p}:=\varinjlim\limits_{\overline{\Q}_p/L/\Q_p}L^\ast
$$ 
which was introduced before \Cref{lemma-character-group-of-pro-torus}. 

\begin{proposition}
\label{proposition-tannakian-group-for-cm-bkf-modules}
The reflex norm (with respect to the ``\'etale realization'')
$$
r\colon \mathrm{BKF}^\circ_{\mathrm{rig},\mathrm{CM}}\to \mathrm{Rep}_{\Q_p}(D_{\Q_p})
$$
is an isomorphism. In particular,
$$
\mathrm{Aut}^\otimes(\omega_{\acute{e}t,\mathrm{CM}})\cong D_{\Q_p}.
$$
\end{proposition}
\begin{proof}
We want to use \Cref{corollary:reflex-norm-defines-equivalence}. In \Cref{lemma:band-of-bkf-modules-connected} with have proven that every rigidified Breuil-Kisin-Fargues modules (up to isogeny) with trivial filtration on its \'etale realization is actually trivial, i.e., a direct sum of the unit object. This implies that the functor $r$ is fully faithful by \Cref{corollary:reflex-norm-defines-equivalence} and then \Cref{theorem-rigidified-bkf-modules-with-cm} implies, using again \Cref{corollary:reflex-norm-defines-equivalence}, that $r$ is essential surjective. This finishes the proof.  
\end{proof}

We remark that \Cref{corollary:reflex-norm-defines-equivalence} implies that the category $\mathrm{BKF}^\circ_{\mathrm{rig},\mathrm{CM}}$ of rigidified Breuil-Kisin-Fargues modules up to isogeny admitting CM is generated by Breuil-Kisin-Fargues modules $(M,\varphi_M)$ whose associated $B_{\dR}^+$-lattice $\Xi\subseteq T\otimes_{\Z_p}B_\dR$ is minuscule, i.e.,
$$
\xi (T\otimes_{\Z_p}B^+_\dR)\subseteq \Xi\subseteq T\otimes_{\Z_p}B^+_{\dR}.
$$ 
By \cite[Proposition 20.1.1.]{scholze_weinstein_lecture_notes_on_p_adic_geometry} these Breuil-Kisin-Fargues modules are associated to $p$-divisible groups over $\mathcal{O}_C$. In particular, we see that the category $\mathrm{BKF}^\circ_{\mathrm{rig},\mathrm{CM}}$ of CM Breuil-Kisin-Fargues modules is generated (as a tensor category) by Breuil-Kisin-Fargues modules associated with $p$-divisible groups admitting CM. This is analogous to the case of rational Hodge structures admitting CM (cf.\ \cite{abdulali_hodge_structures_of_cm_type}). 

\bibliography{/home/ja/Desktop/diverses/biblio/biblio.bib}
\bibliographystyle{plain}

\end{document}